\documentclass[11pt, letterpaper,reqno]{amsart}
\usepackage[left=1in,right=1in,bottom=0.8in,top=0.8in]{geometry} 
\usepackage{amsmath, amssymb,mathtools}
\usepackage{graphicx}
\usepackage{color}
\RequirePackage[colorlinks,citecolor=blue,urlcolor=blue]{hyperref}
\usepackage{amsfonts,amsthm, amscd}
\usepackage{bm}
\usepackage{hyperref}
\usepackage{mathrsfs}
\usepackage{yfonts}
\usepackage{verbatim}
\usepackage{graphicx}
\usepackage[utf8]{inputenc}
\usepackage{latexsym}
\usepackage{lscape}
\usepackage{epsfig}
\usepackage{subfigure}
\usepackage{dsfont}
\usepackage{mathdots} 
\usepackage{upgreek}
\usepackage{tikz} 
\usetikzlibrary{patterns} 
\usetikzlibrary{arrows}
\usetikzlibrary{decorations.markings} 
		\usetikzlibrary{shapes}  
\usepackage[font=small,labelfont=bf]{caption}
\usepackage{epstopdf}
\usepackage{xcolor} 
\usepackage{float} 
\usepackage{enumitem}
\usepackage{pgfplots}
\usepackage{listings}
\usepackage{longtable}  
\usepackage{tipa}
\usepackage{xfrac}  
\usepackage[normalem]{ulem}

\usepackage[only,llbracket,rrbracket]{stmaryrd}

\theoremstyle{plain}
\newtheorem{theorem}{Theorem}[section]
\newtheorem{lemma}[theorem]{Lemma}
\newtheorem{proposition}[theorem]{Proposition}
\newtheorem{corollary}[theorem]{Corollary}

\newtheorem{question}[theorem]{Question}
\theoremstyle{definition}

\newtheorem{remark}[theorem]{Remark}

\numberwithin{equation}{section}
\hypersetup{%
  colorlinks=true,
  urlbordercolor={1 1 1},
  linkcolor={teal},
  raiselinks=false
}
\colorlet{lgray}{white!80!black}
\pgfplotsset{compat=1.9} 
\DeclareFontFamily{U}{mathx}{}
\DeclareFontShape{U}{mathx}{m}{n}{<-> mathx10}{}
\DeclareSymbolFont{mathx}{U}{mathx}{m}{n}
\DeclareMathAccent{\widehat}{0}{mathx}{"70}
\DeclareMathAccent{\widecheck}{0}{mathx}{"71} 
\newcommand{\upperRomannumeral}[1]{\uppercase\expandafter{\romannumeral#1}} 
 
\def \one {\mathds{1}}
\def \be {\begin{equation}}
\def \ee {\end{equation}}
\def \lb {\left(}
\def \rb {\right)}

\def \RR {\mathbb{R}} 
\def \ZZ {\mathbb{Z}}  
\def \Z {\mathbb{Z}}  
\def \N {\mathbb{N}}  
\def \NN {\mathbb{N}}  
\def \ra {\rightarrow}

\def \ral {\longrightarrow}
\def \lal {\longleftarrow}
  
\def \bI {\mathbf{I}}
\def \ha {2}
\def \a {\alpha}
\def \b {\beta}
\def \g {\gamma} 
\def \m {\mu}

\def \ll {\langle}
\def \rr {\rangle}
\def \D {\mathbf{D}}
\def \E {\mathbf{E}} 
\def \A {\mathbf{A}}
\def \hZ {\widehat{Z}}
\def \ld {\sigma_{\ell}}
\def \rd {\sigma_{r}}
\def \dd {\mathsf{d}}
\def \ep {\varepsilon}
\def \lc {\mathbf{loc}}
\def \N {\mathbb{N}}
\def \lBr {\llbracket}
\def \rBr {\rrbracket}

\renewcommand{\P}{\mathbb{P}}
\newcommand{\TV}[1]{{\left\lVert #1 \right\rVert}_{\normalfont
\text{TV}}}

\DeclareMathOperator{\dehp}{DEHP}
\def \mt {\widehat{\mu}}

\usepgflibrary{fpu}
\makeatletter
\let\NAT@parse\undefined
\makeatother
\title{The open ASEP with light particles}    
\author{Dominik Schmid}
\address
{
Dominik Schmid, 
University of Bonn, Germany 
}
\email{d.schmid@uni-bonn.de} 
\author{Zongrui Yang}
\address
{
Zongrui Yang, 
Columbia University, USA 
}
\email{zy2417@columbia.edu} 
\begin{document}
\begin{abstract}
We consider the stationary measure of the open asymmetric simple exclusion process (ASEP) with light particles. We prove several results on the asymptotic locations of the light particles under the stationary measure. Moreover, in the fan region of the high and low density phases, we study the mixing times for open ASEP with light particles. Our approach involves the matrix product ansatz, Askey--Wilson signed measures, moderate deviation results for second class particles, as well as various coupling arguments. Along the way we obtain the particle densities at single sites for the standard open ASEP without light particles, which are of independent interest.
\end{abstract}
\maketitle  
\section{Introduction and main results}
\subsection{Preface}\label{subsec:preface}
The open asymmetric simple exclusion process (ASEP) is a paradigmatic model for non-equilibrium systems with open boundaries and for Kardar--Parisi--Zhang (KPZ) universality. Extensive studies have been devoted to understanding its stationary measure from various perspectives; see \cite{liggett1999stochastic, corwin2022some, blythe2007nonequilibrium, corteel2011tableaux} for a selection of surveys in statistical physics, probability, and combinatorics. While most research has focused on the open ASEP with a single species of particles, progress has also been made in understanding the open ASEPs with two or more species of particles. These systems are known to manifest richer and more interesting physical properties, but they are relatively less understood. 
 
We will focus on the two-species open ASEP with the so-called ``semi-permeable" boundaries, which we refer to in this article as the ``open ASEP with light particles." The ``semi-permeable" boundary conditions allow only the ordinary (or first class) particles to enter or exit at the open boundaries, so in particular, the light (or second class) particles are conserved in the system. Phase transitions as well as other physical properties of the stationary measure of this model were discovered in \cite{arita2006exact,arita2006phase,uchiyama2008two,ayyer2009two,cantini2017asymmetric}. {See Sections \ref{subsec:related works} and \ref{subsec:scope perspective} for a more detailed discussion.} For combinatorial aspects of this model, we refer the reader to \cite{corteel2017combinatorics,corteel2018macdonald,cantini2017asymmetric}. We will investigate the following natural question:
\begin{question}\label{question}
    Consider the stationary measure of the open ASEP on the lattice $\lBr N \rBr:=\{1,\dots,N\}$ with $r$ light particles, which we denote by $\mt_{N,r}$. We fix $r$ and take $N\rightarrow\infty$, then under the stationary measure $\mt_{N,r}$, what are the asymptotic locations of those $r$ light particles in the system?
\end{question}

In this paper, we contribute some partial answers to this question. Our proofs combine several {ideas} from different backgrounds {that, at a higher level, unify the dynamical and stationary perspectives. See Section \ref{subsec:overview of the proof ideas} for more details on the proof techniques.} When $r=1$, we first use the matrix product ansatz \cite{derrida1993exact,uchiyama2008two} to derive a surprising and simple relation between the location of the light particle under the stationary measure $\mt_{N,1}$ and the densities (at single sites) of the stationary measure $\mu_N$ of the standard open ASEP. We then investigate these densities in the open ASEP, which are of independent interest. {The phase diagram (Figure~\ref{fig:phase}) for the open ASEP is determined by the parameters $A$ and $C$, defined in~\eqref{eq:defining ABCD}.}
Everywhere in the phase diagram except for the ``coexistence line," the densities are computed using techniques involving the so-called Askey–Wilson polynomials and Askey–Wilson signed measures \cite{uchiyama2004asymmetric,bryc2010askey,bryc2017asymmetric,wang2023askey}. 
On the coexistence line $A=C>1$, the densities are obtained by combining coupling techniques, the ``local convergence'' result from \cite{bahadoran2006convergence} for the full-line ASEP, and the macroscopic density profile in \cite{wang2023askey}. 
For a general, but finite number $r$ of light particles, we are able to answer Question \ref{question} in the fan region $AC\leq1$ part of the high and low density phases. Here, we rely on recent moderate deviation results from \cite{landon2023tail} for a second class particle on the full-line ASEP with Bernoulli product initial data.  By using the basic coupling and the so-called microscopic concavity coupling from \cite{balazs2009fluctuation}, we transfer these results to establish bounds on the locations of light particles in the open ASEP. We then employ a multi-scale argument to push the light particles step by step into a finite window. Utilizing these estimates, we also derive results on the mixing time of the open ASEP with light particles, using recent results from \cite{gantert2023mixing} on the mixing time of the standard open ASEP.

 
\subsection{Definition of the model}\label{subsec:def of the model}
The open asymmetric simple exclusion process (ASEP) with light particles, which is also known in some other works as the two-species open ASEP with semi-permeable boundaries, is a continuous-time Markov process on $\{0, 1, \ha\}^N$ with parameters 
    \be\label{eq:conditions open ASEP}\a,\b>0,\quad\g,\delta\geq0,\quad0\leq q<1,\ee
    which models the evolution of particles on the lattice $\lBr N \rBr=\{1,\dots,N\}$. There are two types of particles in the system: ``ordinary'' and ``light''. We denote an ordinary particle as $1$ and a light particle as $\ha$, and we also denote a hole as $0$. 
    Particles move in the bulk with the following rates:
    $$\ha 0 \overset{1}{\ral}0\ha,\quad\quad 1 0 \overset{1}{\ral}01,\quad\quad 1\ha \overset{1}{\ral}\ha1,$$
    $$\ha 0 \overset{q}{\lal}0\ha,\quad\quad 1 0 \overset{q}{\lal}01,\quad\quad 1\ha \overset{q}{\lal}\ha1.$$
    At the system's open boundaries, only ordinary particles are allowed to enter or exit. The light particles (of species $\ha$)  are prohibited from entering or exiting. Specifically, at the left boundary, (ordinary) particles enter the system with rate $\alpha$ and exit the system with rate $\gamma$; similarly, at the right boundary, (ordinary) particles enter the system with rate $\delta$ and exit the system with rate $\beta$. The entry of an (ordinary) particle is prohibited if the target site is already occupied, whether by an (ordinary) particle or a light particle. See Figure \ref{fig:openASEP} for an illustration of the jump rates.

 \begin{figure}
\centering
 \begin{tikzpicture}[scale=0.98]
\draw[thick] (-1.2, 0) circle(1.2);
\draw (-1.2,0) node{reservoir};
\draw[thick] (0, 0) -- (12, 0);
\foreach \x in {1, ..., 12} {
	\draw[gray] (\x, 0.15) -- (\x, -0.15);
}
\draw[thick] (13.2,0) circle(1.2);
\draw(13.2,0) node{reservoir};
\fill[thick] (1, 0) circle(0.2);
\fill[thick] (5, 0) circle(0.2);
\fill[thick] (4, 0) circle(0.2);
\fill[thick] (7, 0) circle(0.2);
\draw[thick, ->] (2, 0.3)  to[bend left] node[midway, above]{$1$} (3, 0.3);
\draw[thick, ->] (5, 0.3)  to[bend right] node[midway]{$\times$} (4, 0.3);
\draw[thick, ->] (7, 0.3) to[bend left] node[midway, above]{$1$} (8, 0.3);
\draw[thick, ->] (10, 0.3) to[bend left] node[midway, above]{$1$} (11, 0.3);
\draw[thick, ->] (10, 0.3) to[bend right] node[midway, above]{$q$} (9, 0.3);
\draw[thick, ->] (-0.1, 0.5) to[bend left] node[midway, above]{$\alpha$} (0.9, 0.4);
\draw[thick, <-] (0, -0.5) to[bend right] node[midway, below]{$\gamma$} (0.9, -0.4);
\draw[thick, ->] (12, -0.4) to[bend left] node[midway, below]{$\delta$} (11, -0.5);
\draw[thick, <-] (12, 0.4) to[bend right] node[midway, above]{$\beta$} (11.1, 0.5);
\node[gray] at (1,-0.4) {\tiny $1$};\node[gray] at (2,-0.4) {\tiny $2$};\node[gray] at (3,-0.4) {\tiny $3$};\node[gray] at (4,-0.4) {\tiny $4$};\node[gray] at (5,-0.4) {\tiny\dots};\node[gray] at (11,-0.4) {\tiny $n$};

\node[shape=circle,scale=1.1,fill=white,draw] (E) at (2,0){} ; 
\node[shape=star,star points=5,star point ratio=2.5,fill=black,scale=0.45] (Y1) at (2,0) {};

\node[shape=circle,scale=1.1,fill=white,draw] (E) at (8,0){} ; 
\node[shape=star,star points=5,star point ratio=2.5,fill=black,scale=0.45] (Y1) at (8,0) {};

\node[shape=circle,scale=1.1,fill=white,draw] (E) at (10,0){} ; 
\node[shape=star,star points=5,star point ratio=2.5,fill=black,scale=0.45] (Y1) at (10,0) {};

\end{tikzpicture} 
\caption{Jump rates in the open ASEP with light particles. The black nodes represent ordinary particles and nodes with stars represent light particles. }
\label{fig:openASEP}
\end{figure}
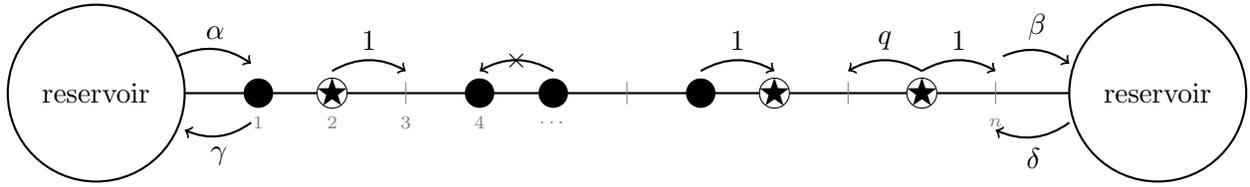 

One can observe that the number $r$ of light particles is conserved. Therefore the ``open ASEP with light particles'' model has $N+1$ irreducible sectors, indexed by the number $r\in \lBr 0,N \rBr:= \{0,1,\dots,N\}$ of light particles. When $r=0$, this model coincides with the standard open ASEP. 

We will work with a re-parameterization of the model by $A,B,C,D$ and $q$, where
\be\label{eq:defining ABCD}
A=\phi_+(\beta,\delta),\quad B=\phi_-(\beta,\delta),\quad C=\phi_+(\alpha,\gamma),\quad D=\phi_-(\alpha,\gamma), 
\ee 
and
\be \label{eq:phi}
\phi_{\pm}(x,y)=\frac{1}{2x}\lb1-q-x+y\pm\sqrt{(1-q-x+y)^2+4xy}\rb, \quad \mbox{for }\; x>0\mbox{ and }y\geq 0.
\ee  

One can check that \eqref{eq:defining ABCD} gives a bijection between \eqref{eq:conditions open ASEP} and 
\be\label{eq:conditions qABCD}
A,C\geq0,\quad -1<B,D\leq 0,\quad 0\leq q<1.
\ee 
We will assume \eqref{eq:conditions open ASEP} and consequently,  \eqref{eq:conditions qABCD} throughout the paper. 
For the standard open ASEP, the quantities $\rho_{\ell}:=1/(1+C)$ and $\rho_r:=A/(1+A)$ defined by the parameters above have nice physical interpretations as the ``effective densities'' near the left and right boundaries of the system.

\begin{figure}[ht]
    \centering
    \begin{tikzpicture}[scale=0.95]
 \draw[scale = 1,domain=6.7:11,smooth,variable=\x,dotted,thick] plot ({\x},{1/((\x-7)*1/3+2/3)*3+5});
 \draw[->] (5,5) to (5,10.2);
 \draw[->] (5.,5) to (11,5);
   \draw[dashed] (5,8) to (8,8);
  \draw[dashed] (8,8) to (8,5);
   \draw[dashed] (8,8) to (10.2,10.2);
   \node [left] at (5,8) {\small$1$};
   \node[below] at (8,5) {\small $1$};
     \node [below] at (11,5) {$A$};
   \node [left] at (5,10) {$C$};
 \draw[dashed] (8,4.9) to (8,5.1);
  \draw[dashed] (4.9,8) to (5.1,8);
 \node [below] at (5,5) {\scriptsize$(0,0)$};
    \node [above] at (6.5,8.5) {LD};\node [above] at (8,8.7) {LD};
    \node at (10.6,10) {\small coexistence line }; 
    \node [below] at (10,6) {HD};  \node [below] at (10,8) {HD}; 
 \node [below] at (6.5,6) {MC}; 
 \node at (6.4,10) {\small{fan}};
 \node at (7.6,10) {\small{shock}};
 \end{tikzpicture} 
    \caption{Phase diagrams for the open ASEP stationary measures. LD, HD, MC respectively
stand for the low density, high density and maximal current phases.}
    \label{fig:phase}
\end{figure}
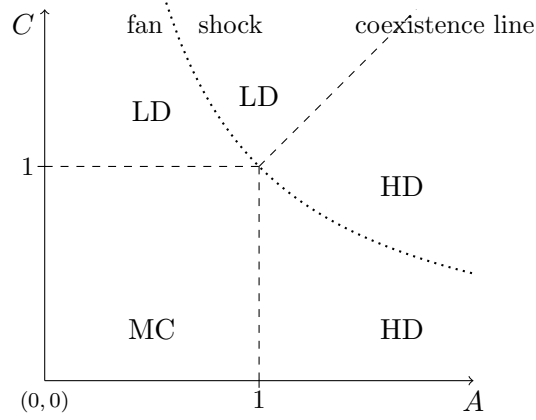

It is known since \cite{derrida1993exact} that the phase diagram of open ASEP involves only two
boundary parameters $A$ and $C$, which exhibits three phases:  
\begin{enumerate}
    \item [$\bullet$] (maximal current phase) $A<1$, $C<1$,
    \item [$\bullet$] (high density phase) $A>1$, $A>C$,
    \item [$\bullet$] (low density phase) $C>1$, $C>A$.
\end{enumerate} 
The boundary $A=C>1$ between the high and low density phases is called the coexistence line.

There are also two regions in the phase diagram distinguished by \cite{derrida2002exact,derrida2003exact}:
\begin{enumerate}
    \item [$\bullet$] (fan region) {$AC\leq1$,}
    \item [$\bullet$] (shock region) $AC>1$.
\end{enumerate}
See Figure \ref{fig:phase} for an illustration.

{
We briefly recall the heuristic behavior of the open ASEP in the different regions of the phase diagram; see also the survey paper~\cite{corwin2022some}. 
In the maximal current phase, particles enter from the left boundary and exit through the right boundary at sufficiently fast rates, and the system achieves the maximal possible transport of particles from left to right. 
In the high-density phase, particles enter from the left boundary at a fast rate but exit through the right boundary at a slow rate, leading to congestion and a buildup of particles near the right boundary.
In the low-density phase, the situation is reversed: particles enter from the left boundary at a slow rate and exit through the right boundary at a fast rate, so that particles do not build up. In the fan region, the associated Burgers equation exhibits a rarefaction fan, while in the shock region it develops a shock.
}
 
\subsection{Location of the light particles}
We denote the state space of the open ASEP with $r$ light particles on the lattice $\lBr N \rBr$ by 
\[
\widehat{\Omega}_{N,r}:=\left\{\tau=(\tau_1,\dots,\tau_N)\, \colon \, \tau_i\in\{0, 1, \ha\} \mbox{ for } i\in\lBr N \rBr,\quad \sum_{i=1}^{N} \mathds{1}{\{ \tau_i = \ha \} } = r\right\}.
\] 
and by $\mt_{N,r}$ the unique stationary measure of the system, which is a probability measure on $\widehat{\Omega}_{N,r}$. 
We denote the locations of the $r$ light particles within the lattice by 
\[
1\leq\lc_1<\dots<\lc_r\leq N.
\]
Next we will introduce our main results, which characterize the distribution of the locations of the light particles under the stationary measure $\mt_{N,r}$, as the system size $N$ approaches infinity. 

We start with the case of a single light particle in the maximal current phase. For integers $a\leq b$, we will write $\llbracket a,b\rrbracket$ to denote $\mathbb{Z}\cap[a,b]$.

\begin{theorem}\label{thm:main thm}
     Assume \eqref{eq:conditions qABCD} and consider the maximal current phase $\max(A,C) \leq 1$ with $AC \neq 1$. Let $r=1$. Suppose   $(a_N)_{N=1,2,\dots}$ and $(b_N)_{N=1,2,\dots}$ are two arbitrary sequences of integers such that  $1\leq a_N<b_N\leq N$ and $a_N,N-b_N\ra\infty$.   Then we have
        \[\lim_{N\ra\infty}\mt_{N,1}\lb\lc_1\in\lBr a_N,b_N \rBr \rb=0, \] 
        \[\lim_{N\ra\infty}\mt_{N,1}\lb\lc_1\in \lBr 1,a_N\rBr \rb=\frac{\ld-1/2}{\ld-\rd},\]
        \[\lim_{N\ra\infty}\mt_{N,1}\lb\lc_1\in \lBr b_N,N\rBr \rb=\frac{1/2-\rd}{\ld-\rd},\]
        where 
        \[\ld:=\frac{3-C-D-CD}{4(1-CD)}>\frac{1}{2}\quad\mbox{and}\quad\rd:=\frac{1 +A+B-3AB}{4(1-AB)}<\frac{1}{2}.\] 
\end{theorem}

The above result will be proved in Section \ref{subsec: 3.2}.  

Next, we present the following result for the high density and low density phases.

\begin{theorem}\label{thm:mainHighLow}
     Assume \eqref{eq:conditions qABCD} and consider the fan region where $AC \leq 1$.  Suppose $(a_N)_{N=1,2,\dots}$ and $(b_N)_{N=1,2,\dots}$ are two arbitrary sequences of integers such that  $1\leq a_N<b_N\leq N$ and $a_N,N-b_N\ra\infty$. Then we have
    \begin{enumerate}
        \item [$\bullet$] In the high density phase $A\geq 1$, $A>C$,  
        \[
        \lim_{N\ra\infty}\mt_{N,r}\lb\lc_1,\dots,\lc_r\in \lBr 1,a_N\rBr \rb=1.
        \] 
        \item [$\bullet$] In the low density phase $C\geq 1$, $C>A$,  
        \[
        \lim_{N\ra\infty}\mt_{N,r}\lb\lc_1,\dots,\lc_r\in \lBr b_N,N\rBr \rb=1.
        \] 
\end{enumerate}
Moreover, when $ABCD\notin\{q^{-l}:l\in\NN_0\}$ and $r=1$, the above statements for the high density and low density phases also hold in the shock region where $AC >1$.
\end{theorem}

The proof of the above theorem is separated into two parts, using entirely different techniques. When $r=1$, $ABCD\notin\{q^{-l}:l\in\NN_0\}$ and $AC\neq1$, we prove the theorem in Section \ref{subsec: 3.2}. For the remaining cases, the proof is located at the end of Section \ref{sec:EscapeLight}.

\begin{remark}\label{rem:extend to shock region}
We conjecture that the requirement of $AC \leq 1$ in Theorem \ref{thm:mainHighLow} can be removed using a suitable local approximation result of the stationary measure of the standard open ASEP by a Bernoulli product measure, similar to \cite{nestoridi2023approximating,yang2024limits}. We leave this point for future research.  
\end{remark}

On the coexistence line, we have the following result in the case of a single light particle. 

\begin{theorem}\label{thm:mainCoexistence}
     Assume \eqref{eq:conditions qABCD}, $ABCD\notin\{q^{-l}:l\in\NN_0\}$ and consider
    coexistence line $A=C>1$. Let $r=1$, then the distribution of $\lc_1/N$ weakly converges to the uniform distribution on $[0,1]$.  
\end{theorem}

The above result will be shown at the end of Section \ref{sec:proof of coexistence line density}. We observe from this result that, on the coexistence line, the asymptotic location of a single light particle is the same (uniform distribution) as the asymptotic location of the shock in the standard open ASEP, see Remark \ref{rmk:shock} for details.

Our three main theorems each provide partial answers to Question \ref{question} for different values of $r$ and within different parts of the phase diagram. When $r=1$, the distribution of the location of the light particle is characterized throughout the phase diagram, except at the `triple point' $(A,C)=(1,1)$, where our methods do not apply. We conjecture that at the triple point, for $r=1$, the normalized location of the light particle converges to the uniform distribution on $[0,1]$, i.e., it has the same asymptotics as on the coexistence line. We leave the investigation of this point for future research. As mentioned in Section \ref{subsec:preface}, the proofs of the main theorems utilize very different techniques for the case $r=1$, and for $r \geq 1$ in the fan region of the high and low density phases.


\subsection{Mixing times for the open ASEP with light particles in the high and low density phases}

For the open ASEP with light particles, we are also interested in quantifying the speed of convergence to the stationary measure $\widehat{\mu}_{N,r}$. To do so, we consider the \textbf{total-variation distance}
\begin{equation} \label{def:TVdistance}
\TV{ \nu - \nu^{\prime}} := \max_{A \subseteq \Omega} \lb\nu(A) - \nu^{\prime}(A)\rb
\end{equation}
 between two probability measures $\nu$ and $\nu^{\prime}$ on a common finite {state} space $\Omega$. 
 We define the total-variation \textbf{mixing time} for open ASEP with light particles on state space $\widehat{\Omega}_{N,r}$ for all $\varepsilon \in (0,1)$ as
 \begin{equation}
     t^N_{\textup{mix}}(\varepsilon) := \inf \left\{ t \geq 0 \, \colon \, \max_{\tilde{\tau} \in \widehat{\Omega}_{N,r}} \TV{ \P(\tau(t) \in \, \cdot \, | \, \tau(0)=\tilde{\tau}) - \widehat{\mu}_{N,r} } \leq \varepsilon \right\},
 \end{equation}
where $\tau(t)=\lb\tau_1(t),\dots,\tau_N(t)\rb$ denotes the occupation variables for the process at time $t\geq0$. 
 A general introduction to mixing times for Markov chains can be found in \cite{LPW:markov-mixing}. 
We have the following result on the mixing time in the fan region of the high and the low density phases.
\begin{theorem}\label{thm:MainMixingTime}
Assume \eqref{eq:conditions qABCD} and $A,C>0$. Consider the fan region $AC \leq 1$ of either the high density phase where {$A > 1$,} or the low density where {$C > 1$.} Consider the mixing time of the open ASEP model with $r$ light particles on the lattice $\lBr N \rBr$.
Then there exist constants $c_0,C_0>0$, depending only on $A$ and $C$,  such that for every $\varepsilon\in(0,1)$,
\begin{equation}\label{eq:MixingEquation}
   c_0 \leq  \liminf_{N \rightarrow \infty} \frac{t^N_{\textup{mix}}(\varepsilon)}{N} \leq   \limsup_{N \rightarrow \infty} \frac{t^N_{\textup{mix}}(\varepsilon)}{N} \leq C_0 . 
\end{equation} 
\end{theorem}
The fact that the constants $c_0,C_0$ in \eqref{eq:MixingEquation} do not depend on the choice of $\varepsilon$ is also known as \textbf{pre-cutoff}. Let us remark that Theorem \ref{thm:MainMixingTime} is compatible with Theorem 1.5 in \cite{gantert2023mixing}, which states that the mixing time of the standard open ASEP  in the high and the low density phase is of the same order $N$. In the maximal current phase and the coexistence line, there are to our best knowledge no mixing time results available for the open ASEP without light particles, apart from the special case of the open TASEP \cite{schmid2022mixing,schmid2023mixing,elboim2024mixing} and the triple point \cite{gantert2023mixing}.


 
\subsection{Density of the standard open ASEP}\label{subsec:density of standard open ASEP}
In the proofs of {Theorems~\ref{thm:main thm}, \ref{thm:mainCoexistence}, and the $r=1$ case of Theorem~\ref{thm:mainHighLow},} which concern the open ASEP with a single light particle, we need to obtain several results related to the limits of particle densities in the standard open ASEP. We believe these results have independent interest. They can be regarded as the ``microscopic limits'' of the particle density, focusing solely on the densities at a single site, as opposed to the ``macroscopic limits'' of particle densities discussed in \cite{wang2023askey}. The latter concern the numbers of particles in intervals of length proportional to $N$.

The next result concerns the particle densities of the open ASEP at the bulk of the system, everywhere in the phase diagram except for the coexistence line $A=C>1$.


\begin{proposition}
[c.f.\ Remark A.3 in \cite{nestoridi2023approximating}] \label{prop:bulk density}
Assume \eqref{eq:conditions qABCD}. Suppose $(a_N)_{N=1,2,\dots}$ is an arbitrary sequence of integers such that $a_N,N-a_N\ra\infty$. {We denote by $\mu_N$ the stationary measure of the open ASEP.} We have 
    \be\lim_{N\rightarrow\infty}\m_N(\tau_{a_N}=1)=\begin{cases}
   A/(1+A)  & \quad A>C, A>1 \quad\text{ (high density)}, \\
   1/(1+C) & \quad C>A, C>1 \quad\text{ (low density)}, \\
   1/2 & \quad A\leq1, C\leq1 \quad\text{ (maximal current)}.
\end{cases}\ee
 See Figure \ref{fig:density bulk} for an illustration of these densities.
\end{proposition}

This result can be regarded as a corollary of \cite[Theorem 3.29]{liggett1999stochastic}. However, it is presented therein only in the open TASEP case where $q=0$. For a more detailed explanation of how the result extends to $q\in(0,1)$, we refer the reader to Proposition A.2 and Remark A.3 in \cite{nestoridi2023approximating}.

\begin{figure}  
    \begin{centering}
    \begin{tikzpicture}[scale=0.95]
 \draw[->] (5,5) to (5,10.2);
 \draw[->] (5.,5) to (11.2,5);
   \draw[-, dashed] (5,8) to (8,8);
   \draw[-, dashed] (8,8) to (8,5);
   \draw[- ] (8,8) to (10,10);
   \node [left] at (5,8) {\scriptsize$1$};
   \node[below] at (8,5) {\scriptsize $1$};
     \node [below] at (11,5) {$A $};
   \node [left] at (5,10) {$C $};  
  \node [above] at (4,8.7) {\small$\frac{1}{1+C}$};  
  \draw[-] (8,4.9) to (8,5.1);
   \draw[-] (4.9,8) to (5.1,8); 
 \node [below] at (5,5) {\scriptsize$(0,0)$}; ;
     \draw[->] (4.5,9) to [out=-45,in=-135] (6,9);
      \node [below] at (10,4.6) {\small$\frac{A}{1+A}$}; 
       \draw[->] (10,4.5) to [out =45,in= -45] (10,6) ;
           \node [above] at (6.5,9) {LD}; 
    \node [below] at (10,6.5) {HD};  
     \node [below] at (6.5,6.5) {MC};  
     \draw[->] (4.5,6) to [out=-45,in=-135] (6,6.08);
  \node [above] at (4.3,5.75) {\small$\frac{1}{2}$}; 
  \draw[->] (4.5,8) to [out=-45,in=-135] (6,8);
  \node [above] at (4.3,7.7) {\small $\frac{1}{2}$}; 
  \node[below] at (7.9,4.6) {\small$\frac{1}{2}$}; 
 \draw[->] (8,4.4) to [out =45,in= -45] (8,6) ;
 \draw[->] (10,8.5) to [out=-90,in=-35] (8,8);
  \node at (10.1,8.7) {\small $\frac{1}{2}$};  
\end{tikzpicture} 
    \caption{Densities at single sites at the bulk of open ASEP, as stated in Proposition~\ref{prop:bulk density}. The densities are shown in the figure everywhere in the phase diagram, except along the coexistence line $A=C>1$, where a different phenomenon occurs, as described in Proposition \ref{prop:coexistence}.  We also note that the densities coincide with the ``macroscopic densities,'' which pertain to the number of particles in macroscopic intervals of length proportional to the system size $N$; see, for example,  \cite[Figure 2 (a)]{wang2023askey}.}
    \label{fig:density bulk}\end{centering}
\end{figure}
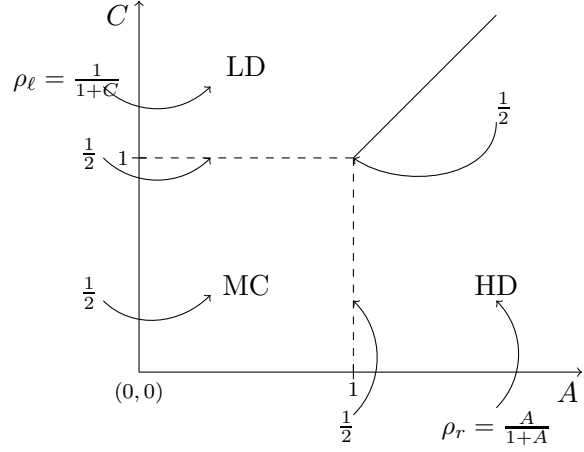


We next provide the densities at the bulk of the system along the coexistence line $A=C>1$.

\begin{proposition}\label{prop:coexistence}
Assume \eqref{eq:conditions qABCD} and $ABCD\notin\{q^{-l}:l\in\NN_0\}$. Consider the coexistence line $A=C>1$. Suppose  $(a_N)_{N=1,2,\dots}$ is an arbitrary sequence of integers such that $1\leq a_N\leq N$ and $a_N/N\rightarrow\theta$ for some $\theta\in[0,1]$. We have 
    \begin{equation}\label{eq:veerwf}
    \lim_{N\rightarrow\infty}\mu_N(\tau_{a_N}=1)= (1-\theta)\frac{1}{1+A}+\theta\frac{A}{1+A}.
    \end{equation}
\end{proposition}
This proposition will be proved in Section \ref{sec:proof of coexistence line density}, where we use some coupling arguments, a local convergence result by Bahadoran and Mountford \cite{bahadoran2006convergence} as well as a result on the macroscopic density profile of open ASEP from \cite{wang2023askey}.
We mention that the density profile \eqref{eq:veerwf} on the coexistence line being linear with respect to location $\theta$ has been observed in the physics literature \cite{derrida1993exact,schutz1993phase,essler1996representations,mallick1997finite,derrida2002exact,derrida2003exact}. Our result confirms the postulation in these works.

The following result concerns the particle density of the open ASEP at the open boundaries.
\begin{proposition}\label{prop:boundary density}
  Assume \eqref{eq:conditions qABCD}. As $N\ra\infty$, the probabilities $\m_N(\tau_1=1)$ and $\m_N(\tau_N=1)$ converge respectively to $\ld:=\ld(A,B,C,D)$ and $\rd:=\rd(A,B,C,D)$ given by the following formulas:
\begin{enumerate}
 \item [$\bullet$] In the maximal current phase $A\leq1$ and $C\leq1$, we have 
 \[\ld=\frac{3-C-D-CD}{4(1-CD)},\quad\rd=\frac{1 +A+B-3AB}{4(1-AB)}.\]
 \item [$\bullet$] In the high density phase $A\geq 1$ and $A\geq C$, we have 
 \[\ld=\frac{A^2+A+1-ACD-AC-AD}{(A+1)^2(1-CD)},\quad\rd=\frac{A}{1+A}.\]
\item [$\bullet$] In the low density phase $C\geq 1$ and $C\geq A$, we have 
\[\ld=\frac{1}{1+C},\quad\rd=\frac{AC+BC-ABC^2-AB+C-ABC}{(C+1)^2(1-AB)}.\]
\end{enumerate}
We refer to $\ld$ and $\rd$ as the limiting densities at the left and right boundaries, respectively.
\end{proposition}

The above result will be proved in Section \ref{subsec: 3.1}. 
\begin{remark}\label{rmk:left and right density}
Proposition \ref{prop:boundary density} provides the limiting densities $\ld$ and $\rd$ everywhere on the phase diagram. As can be observed, these limiting densities are continuous functions involving $A$, $B$, $C$, and $D$ on the entire paremeter space. Figure \ref{fig:boundary densities} compares the limiting densities at the boundaries, $\ld$ and $\rd$, with the densities at the bulk of the system. The latter are provided by Proposition \ref{prop:bulk density} and Figure \ref{fig:density bulk}. It is worth noting that $\ld = \rd$ holds if and only if $AC=1$. Along this curve $AC=1$, it is well-known that the stationary measure $\mu_N$ is a Bernoulli product measure with probability $\rho_{\ell}=\rho_r$, see for example \cite[Appendix A]{enaud2004large}.
\end{remark}
\begin{figure}
    \centering
    \begin{tikzpicture}[scale=0.95]
 \draw[scale = 1,domain=8:11,smooth,variable=\x,dotted,thick] plot ({\x},{1/((\x-7)*1/3+2/3)*3+5});
 \draw[->] (5,5) to (5,10.2);
 \draw[->] (5.,5) to (11,5);
   \draw[dashed] (5,8) to (8,8);
  \draw[dashed] (8,8) to (8,5);
   \draw[dashed] (8,8) to (10.2,10.2);
   \node [left] at (5,8) {\small$1$};
   \node[below] at (8,5) {\small $1$};
     \node [below] at (11,5) {$A$};
   \node [left] at (5,10) {$C$};
 \draw[dashed] (8,4.9) to (8,5.1);
  \draw[dashed] (4.9,8) to (5.1,8);
 \node [below] at (5,5) {\scriptsize$(0,0)$};   
 \node at (6.6,6.6) {\small $\ld>1/2$};
 \node at (8,9.3) {\small $\ld=\rho_{\ell}$};
 \node at (9.8,6) {\small $\ld>\rho_r$};
 \node at (9.8,8.5) {\small $\ld<\rho_r$};
 \node at (8,3.7) {(a)};
 \end{tikzpicture}\quad\quad\quad\quad
 \begin{tikzpicture}[scale=0.95]
 \draw[scale = 1,domain=6.7:8,smooth,variable=\x,dotted,thick] plot ({\x},{1/((\x-7)*1/3+2/3)*3+5});
 \draw[->] (5,5) to (5,10.2);
 \draw[->] (5.,5) to (11,5);
   \draw[dashed] (5,8) to (8,8);
  \draw[dashed] (8,8) to (8,5);
   \draw[dashed] (8,8) to (10.2,10.2);
   \node [left] at (5,8) {\scriptsize$1$};
   \node[below] at (8,5) {\scriptsize $1$};
     \node [below] at (11,5) {$A$};
   \node [left] at (5,10) {$C$};
 \draw[dashed] (8,4.9) to (8,5.1);
  \draw[dashed] (4.9,8) to (5.1,8);
 \node [below] at (5,5) {\scriptsize$(0,0)$};  
 \node at (10,8) {\small $\rd=\rho_r$};
 \node at (6.6,6.6) {\small $\rd<1/2$};
 \node at (6,8.7) {\small $\rd<\rho_{\ell}$};
 \node at (8,10) {\small $\rd>\rho_{\ell}$};
\node at (8,3.7) {(b)};
 \end{tikzpicture} 
    \caption{A comparison of the limiting densities $\ld$ and $\rd$ at the left and right boundaries (Proposition \ref{prop:boundary density}) with the limiting densities in the bulk of the system (Proposition \ref{prop:bulk density} and Figure \ref{fig:density bulk}). 
    All the inequalities in this figure can be proved by simple computation. 
    Along the boundary curves of the regions in the phase diagrams, the comparison is not shown in the figures but can be easily observed using the fact that $\ld$ and $\rd$ are continuous functions of $A$, $B$, $C$ and $D$. For example, in part (a), we have $\ld=\rho_{\ell}$ along the boundary curves $\{(A,C):A<1,C=1\}$ and $\{(A,C):A=C\geq1\}$ of the low density phase; $\ld=1/2$ on $\{(A,C):A=1,C<1\}$; and $\ld=\rho_r$ on $\{(A,C):AC=1,A>1\}$. In part (b), the comparison of $\rd$ with the bulk densities along boundary curves can be observed similarly. } 
    \label{fig:boundary densities}
\end{figure}
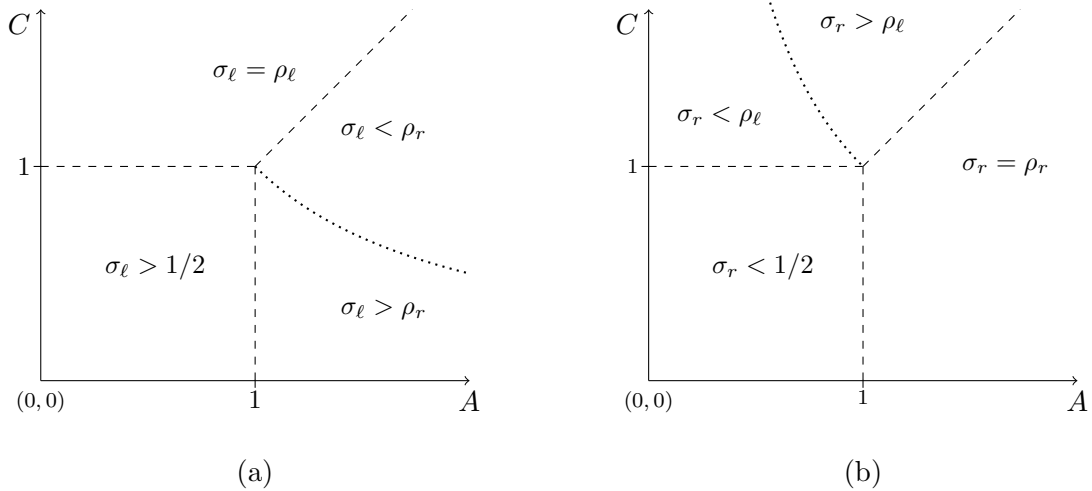

%

\subsection{Related works} \label{subsec:related works}
One of the most important tools for studying the stationary measures of simple exclusion processes with open boundaries is the matrix product ansatz (MPA). This method was introduced by Liggett \cite{liggett1975ergodic} in an implicit form and by Derrida et al. \cite{derrida1993exact} for the open TASEP; {see also \cite{blythe2000exact,sasamoto1999one} for extensions to the open ASEP,} and \cite{blythe2007nonequilibrium,corwin2022some} for introductory surveys. Over the past five decades, the matrix product ansatz has led to numerous studies of the stationary measure of the open ASEP, covering a wide range of asymptotic behaviors, including phase diagrams, density profiles, correlation functions, limit fluctuations, and large deviations. For a non-exhaustive list of related papers, see \cite{derrida1993exact,schutz1993phase,essler1996representations,mallick1997finite,derrida2002exact,derrida2003exact,sasamoto2000density,uchiyama2004asymmetric,uchiyama2005correlation} and the references in \cite{blythe2007nonequilibrium} for physics works, and \cite{bryc2017asymmetric,bryc2019limit,wang2023askey} for mathematical works. It is known from \cite{uchiyama2004asymmetric} (see also \cite{corteel2011tableaux}) that the matrix product ansatz is closely related to the so-called Askey--Wilson orthogonal polynomials introduced by Askey and Wilson \cite{askey1985some}. 
{This connection was first identified in the special case $\gamma=\delta=0$ in earlier work by Sasamoto~\cite{sasamoto1999one}.} More recently, this relationship has led to a characterization of the stationary measure of open ASEP in terms of Askey--Wilson (signed) measures by Bryc and Wesołowski \cite{bryc2010askey,bryc2017asymmetric} in the `fan region' and was later extended by Wang et al. \cite{wang2023askey} to the `shock region'. This approach not only enables rigorous derivations of asymptotic behaviours of the stationary measure of the  open ASEP, but has also led to studies of the stationary measures of the open KPZ equation \cite{corwin2024stationary,bryc2023markov,le2022steady} and the open KPZ fixed point \cite{bryc2023asymmetric,wang2024asymmetric}.
Furthermore, we mention that very recently a new method has been developed to characterize the stationary measures of certain integrable models with two open boundaries, in terms of marginal distributions of the so-called ``two-layer Gibbs measures'', see \cite{barraquand2024stationary,bryc2024twoLayerTASEP,bryc2024stationary,bryc2024limitCoex}.

While the stationary measure for the open ASEP with a single species of particles is relatively well understood, the corresponding questions for open ASEPs with two or more species of particles remain largely unresolved. These systems exhibit richer physical properties, such as spontaneous symmetry breaking \cite{evans1995asymmetric} and condensation \cite{arndt1999spontaneous,evans2000phase}. Moreover, various choices of boundary conditions exist for multi-species models. The integrable boundaries were classified for two-species open ASEPs in \cite{crampe2015open} and for more species models in \cite{crampe2016integrable}; see also \cite{crampe2016matrix,finn2018matrix} for the matrix product ansatz of such integrable models. A classical choice of the boundary condition is the ``semi-permeable" one discussed in the present article, which forbids the light particles from entering or exiting the system. The phase transitions and other stationary properties for TASEP and ASEP versions of such semi-permeable two-species models were studied in \cite{arita2006exact,arita2006phase,uchiyama2008two,ayyer2009two,cantini2017asymmetric}. It is also known from \cite{corteel2017combinatorics,corteel2018macdonald,cantini2017asymmetric} that the stationary measures are related to the so-called Macdonald–Koornwinder polynomials \cite{koornwinder1992askey}, which can be considered as multivariate generalizations of Askey–Wilson polynomials. The ``left-permeable" two-species open ASEP, on the other hand, allows light particles to enter and exit at the left boundary and is studied in \cite{ayyer2018matrix}. For open ASEP with an arbitrary number of species, see, for example, \cite{ayyer2017exact,roy2021phase,ayyer2019phase} for studies of the phase diagrams.

The role of light particles (often referred to as second class particles) is closely related to physical properties of the asymmetric simple exclusion process. For the ASEP on the integers, it is shown in \cite{BS:ExactConnections} that the moments of the location of a second class particle are linked to the moments of the current of the ASEP, i.e.\ the number of particles passing through a given site over time; see also \cite{BS:OrderCurrent,FF:CurrentFluctuations}. However, determining the law of a second class particle remains challenging; see \cite{A:CurrentFluctuations,ACG:ASEPspeed,landon2023tail} for recent progress on the fluctuations and convergence of second class particles for stationary and step initial conditions. 
Only in the special case of the totally asymmetric simple exclusion process, where $q=0$, significant progress could be achieved for general initial conditions due to an alternative representation as a last passage percolation model, where the motion of a light particle can be expressed as a competition interface \cite{FP:CompetitionInterface,GRS:GeodesicCompetition}. 
Another recent line of research concerns the open ASEP started outside of the stationary distribution.  In this case, mixing times are a standard way to characterize the speed of convergence. We refer to \cite{BN:CutoffASEP,LL:CutoffASEP,LL:CutoffWeakly} for recent results on the mixing time of the asymmetric simple exclusion process on a closed segment, and \cite{elboim2024mixing,gantert2023mixing,he2023limit,schmid2022mixing,schmid2023mixing} for the open ASEP. 

{
\subsection{Scope and perspective}\label{subsec:scope perspective}
We explain how the model studied in this paper, together with the results we establish for it, fits within the existing literature. 
Multi-species open ASEP models are natural extensions of the classical open ASEP, in which particles of different types (or species) interact according to a hierarchy. 
Such systems provide richer models of non-equilibrium transport, with applications ranging from traffic flow with vehicles of different speeds that may overtake to molecular transport involving multiple types of motors. 
At the physical level, multi-species open ASEP models have been observed to exhibit more complex behavior than in the single-species case, including phase diagrams with a larger number of phases~\cite{roy2021phase}, the appearance of (nested) fat shocks rather than sharp shocks~\cite{ayyer2009two,ayyer2017exact}, spontaneous symmetry breaking~\cite{evans1995asymmetric}, and condensation phenomena~\cite{arndt1999spontaneous,evans2000phase}. 
Despite sustained interest from both mathematics and physics, aside from a small number of works on phase diagrams, a rigorous understanding of the asymptotic behavior of multi-species open ASEP remains largely open.

The class of multi-species open ASEP models is broad. 
Even in the two-species case, the open ASEP involves several parameters describing the interchange rates between different particle species at the boundaries, and the integrable cases have been classified in~\cite{crampe2015open}. 
In this paper, we focus on the two-species open ASEP with semi-permeable boundaries, meaning that ordinary particles may enter and exit the system, while light (second-class) particles are conserved.
This setting represents the most classical boundary regime for the two-species open ASEP and is commonly treated as a base case for further analysis. 
For example, \cite{ayyer2018matrix} studied left-permeable boundary conditions by relating their matrix product ansatz to that of the semi-permeable case. 
More recently, studies of these base boundary regimes have also been generalized to open ASEPs with a general number of species.
See \cite{ayyer2019phase,ayyer2017exact,finn2018matrix,roy2021phase,pahuja2023correlations} for a non-exhaustive list.

Despite the relative simplicity of the two-species semi-permeable open ASEP within the class of multi-species open ASEP models, rigorous asymptotic results for this system remain limited beyond the phase diagram. 
A key difficulty is that the model is not, in general, an attractive interacting particle system: there is no natural ordering of configurations that is preserved by the dynamics. 
Moreover, unlike the two-species ASEP on the full line or on a ring—which can be realized as a monotone coupling of two single-species ASEPs started from different initial conditions—this coupling interpretation breaks down in the presence of open boundaries due to interactions with the reservoirs. 
Consequently, many standard methods do not apply in the open two-species setting.

The present paper establishes several rigorous results for the two-species semi-permeable open ASEP. 
Most of these results address Question~\ref{question} in different parts of the phase diagram and for different numbers~$r$ of light particles. 
This question concerns the asymptotic locations of the $r$ light particles under the stationary measure and can be viewed as a first step toward understanding the full density profile. 
While existing results describe the phase diagram for the mean particle densities (i.e., the total number of particles divided by~$N$), comparatively little is known about the density profile, namely the spatial distribution of particles within the system. 
Our results provide explicit descriptions of the density profiles for light particles, which exhibit distinct behavior in different phases.
In particular, Theorem~\ref{thm:main thm} gives an explicit formula for the location of a light particle in the maximal current phase, which reveals nontrivial structure in the behavior of the light particle.
The formula is obtained via a simple structural relation (Proposition~\ref{prop:simple relation}) between the open ASEP with one light particle and the standard single-species open ASEP, which allows the distribution of the light particle to be computed from the microscopic density profile (i.e., at single sites) of the single-species model. 
On the coexistence line, Theorem~\ref{thm:mainCoexistence} shows that the light particle is uniformly distributed along the system, matching the distribution of the shock in the single-species open ASEP. 
Our paper also provides, to the best of our knowledge, the first rigorous proof of the linear microscopic density profile for the open ASEP on the coexistence line with general boundary parameters, a behavior long conjectured and extensively used in the physics literature and previously computed under special parameter choices (see, for example,~\cite{derrida1993exact,sasamoto2000density}).

Finally, the present work also initiates a rigorous analysis of dynamical properties of the two-species open ASEP. 
The proof of Theorem~\ref{thm:mainHighLow} for general~$r$ relies on quantitative control of the locations of light particles at finite times, together with a multi-scale argument that progressively drives these particles toward the boundary. 
See Section~\ref{subsec:overview of the proof ideas} for a more detailed discussion of these methods. 
These dynamical estimates further lead to our result on mixing times (Theorem~\ref{thm:MainMixingTime}) in the high- and low-density phases, where we establish a pre-cutoff phenomenon.
Mixing times quantify how quickly a Markov chain forgets its initial state and approaches equilibrium and are a standard object of study.
While mixing times for the standard open ASEP have been investigated in several works and shown to exhibit rich behavior, to our knowledge Theorem~\ref{thm:MainMixingTime} provides the first result on mixing times for the two-species open ASEP.

\subsection{Overview of the proof ideas}\label{subsec:overview of the proof ideas}
Previous asymptotic analyses of the two-species open ASEP have relied primarily on the matrix product ansatz introduced by Uchiyama~\cite{uchiyama2008two}, in some cases supplemented by physical heuristics. 
The present work develops a collection of new techniques for the analysis of this model by combining several ideas and methods originating from different areas.  
At a conceptual level, a key feature of our approach is the unification of dynamical and stationary viewpoints.
The dynamical perspective draws crucially on recent progress in the understanding of ASEP dynamics on the full line, while the stationary perspective makes essential use of the analysis of the stationary measure of the single-species open ASEP. 
Taken separately, neither viewpoint appears sufficient to access the asymptotic results established in this paper.

Our analysis can be roughly divided into two parts. 
A central ingredient of the first part is a structural relation, derived from the matrix product ansatz, between the stationary measure of the two-species semi-permeable open ASEP and that of the single-species open ASEP. 
This relation underlies the proofs of Theorems~\ref{thm:main thm}, \ref{thm:mainCoexistence}, and the $r=1$ case of Theorem~\ref{thm:mainHighLow}, which address Question~\ref{question} for a single light particle. 
This relation also extends to a general number $r$ of light particles (see Remark~\ref{rem:multiple light particles}), suggesting a possible framework for further analysis in this direction. 
When $r=1$, the relation (Proposition~\ref{prop:simple relation}) takes a simple form, directly relating the location of the light particle to differences of microscopic densities (i.e., densities at single sites) of the single-species open ASEP. 
These microscopic densities at the boundaries can be computed (see Section~\ref{sec:proof of boundary density}) using recent advances in methods based on Askey--Wilson polynomials~\cite{bryc2017asymmetric,yang2024limits}, which in turn lead to the proofs of Theorem~\ref{thm:main thm} and the $r=1$ case of Theorem~\ref{thm:mainHighLow}.

On the coexistence line, the relation described above reduces the proof of Theorem~\ref{thm:mainCoexistence} to establishing the linear microscopic density profile of the standard open ASEP (i.e., Proposition~\ref{prop:coexistence}). 
As mentioned, the linear density profile has long been conjectured in the physics literature and computed for special boundary parameters (see, for example, Sasamoto~\cite{sasamoto2000density} for the case $\gamma=\delta=0$).  
While it may be possible, after substantial additional work, to make the asymptotic analysis in~\cite{sasamoto2000density} rigorous and to extend it to general parameters $(q,\alpha,\beta,\gamma,\delta)$, we adopt here a different and more conceptual approach. 
Rather than relying solely on static properties, we use information from the dynamics of the process to infer properties of its stationary measure.

Specifically, in Section \ref{sec:proof of coexistence line density} we adapt a local convergence result of Bahadoran--Mountford~\cite{bahadoran2006convergence} for the full-line ASEP to the setting of open ASEP, using coupling methods. 
This local convergence result roughly states that if the initial condition has asymptotic density $\rho$ in a macroscopic neighborhood (i.e., with length proportional to $N$) of a given site, then after a time of order $N$, the probability of observing a particle at that site is close to $\rho$. 
A careful application of this result allows us to derive the microscopic density profile for the open ASEP on the coexistence line from its macroscopic density profile (i.e., the number of particles in $\llbracket aN,bN\rrbracket$), which was recently proven rigorously in~\cite{wang2023askey}. 
This dynamical argument establishes properties of the stationary measure that appear difficult to obtain using purely static methods, even with recent techniques such as the two-layer Gibbs measures and recent advances of the Askey--Wilson approach (see Section~\ref{subsec:related works}). 
The combination of dynamical and stationary perspectives thus plays a central role in this proof.

The second part of our analysis addresses Question~\ref{question} for a general number~$r$ of light particles and leads to the proof of Theorem~\ref{thm:mainHighLow}. 
The techniques developed in this part further imply Theorem~\ref{thm:MainMixingTime} on mixing times.
Although, as discussed above, the relation between the stationary measure of the two-species open ASEP and that of the standard open ASEP extends to the setting of general~$r$, additional obstacles arise when carrying out asymptotic analysis with multiple light particles; see Remark~\ref{rem:multiple light particles} for a detailed discussion. 
For this reason, we adopt a different approach, relying on a set of techniques distinct from those used in the first part of the analysis. 
Nevertheless, they share a common principle of using dynamical estimates of the process to derive results for the stationary measure, and making essential use of ASEP dynamics on the full line.

The analysis in this part is technical and constitutes Section~\ref{sec:Concentration}; a more detailed overview is provided in Section~\ref{subsec:strategy for the proof}. 
Here we briefly describe the main ideas.
A recent work of Landon and Sosoe~\cite{landon2023tail} establishes moderate deviation estimates for a single light particle in the full-line ASEP: starting from a Bernoulli $\rho$-product initial condition, the location of the light particle moves with a constant speed and exhibits fluctuations of order~$t^{2/3}$ at time~$t$, with a super-exponential tail bound. 
By combining several coupling and stochastic domination arguments, we transfer these results to the open ASEP and establish the constant speed for the rightmost light particle in the high-density phase of the open ASEP with~$r$ light particles, starting from the stationary measure. 
These arguments include the basic coupling in different cases, the microscopic concavity coupling for full-line ASEP with light particles introduced by Balázs and Seppäläinen~\cite{balazs2009fluctuation}, which preserves the ordering of light particles over time, comparisons between full-line and open ASEP via coupling, and stochastic domination of the open ASEP stationary measure by Bernoulli product measures.

Building on the resulting location estimates for the rightmost light particle, we develop a multi-scale argument to progressively drive all light particles into a finite window near the left boundary. 
Specifically, we introduce a hierarchy of space--time partitions at different length scales and control the evolution of light particles step by step, applying the location bounds at each step. 
This iterative procedure allows us to bootstrap the estimates and ultimately confine all light particles to a finite region. 
Multi-scale arguments have previously appeared in last-passage percolation, where they are used to obtain uniform control on transversal fluctuations of geodesics~\cite{basu2014last,basu2021temporal,schmid2022mixing}. 
To our knowledge, the present work is the first instance in which the multi-scale approach is implemented in the setting of exclusion processes. 
Recent moderate deviation bounds of Landon and Sosoe~\cite{landon2023tail} play a crucial role in enabling this argument, and the resulting analysis demonstrates that multi-scale methods, when combined with suitable dynamical estimates, can be effective tools for controlling the behavior of second-class particles in interacting particle systems.

 }





\subsection{Outline of the paper}

This paper is structured as follows. In Section~\ref{sec:MPA}, we introduce the matrix product ansatz and provide the relation between the location of a single light particle under the stationary measure {$\mt_{N,1}$} and the densities of the stationary measure $\mu_N$ of the standard open ASEP without light particles. The respective density results for the standard open ASEP are derived in Section~\ref{sec:proof of boundary density}, allowing us to conclude Theorem~\ref{thm:main thm} and Theorem~\ref{thm:mainHighLow} for $r=1$. Section \ref{sec:Concentration} concerns the locations of light  particles in the fan region of the high and low density phases, as well as mixing times, proving Theorem~\ref{thm:mainHighLow} for $r \geq 1$ and Theorem~\ref{thm:MainMixingTime}. In Section~\ref{sec:proof of coexistence line density}, we characterize the density of open ASEP on the coexistence line and prove Theorem~\ref{thm:mainCoexistence}.

\subsection*{Acknowledgements}
We thank the Simons Center for Geometry and Physics where this work was initiated at the workshop `The asymmetric simple exclusion process'.

\section{A simple relation from the matrix product ansatz}
\label{sec:MPA}
In this section, we use the matrix product ansatz to derive a simple and surprising relation between the stationary measure of the standard open ASEP and the stationary measure of the open ASEP with a single light particle. Combined with the results on the densities of the open ASEP in Section \ref{subsec:density of standard open ASEP}, this relation will eventually allow us to prove our main theorems for the case of the open ASEP with a single light particle.

We denote the unique stationary measure of the standard open ASEP by $\mu_N$. The following result relates the stationary measure $\mt_{N,1}$ of the open ASEP with a light particle to $\mu_{N+1}$.

\begin{proposition}\label{prop:simple relation}
Assume \eqref{eq:conditions qABCD}, $ABCD\notin\{q^{-l}:l\in\NN_0\}$ and $AC\neq1$. For any $1\leq k\leq l\leq N$, 
\be\label{eq:simple relation}
\mt_{N,1}\lb k\leq\lc_1\leq l\rb=\frac{\m_{N+1}(\tau_k=1)-\m_{N+1}(\tau_{l+1}=1)}{\m_{N+1}(\tau_1=1)-\m_{N+1}(\tau_{N+1}=1)}.
\ee
\end{proposition}

This proposition will be shown at the end of this section as a corollary of the matrix product ansatz for both the stationary measures $\mt_{N,1}$ and $\mu_{N+1}$. We remark that this result is interesting to us because it provides a simple relation between the stationary measures of two  probability systems, and yet we cannot provide a direct probabilistic proof; one must go through the algebraic machinery of the matrix product ansatz.

{
\begin{remark}\label{rem:multiple light particles}
Proposition~\ref{prop:simple relation} admits an extension to a general number~$r$ of light particles. 
More precisely, one can show that the probability that the $r$ light particles lie in disjoint intervals $\llbracket k_i,l_i\rrbracket$ for $i=1,\dots,r$ is proportional to the correlation function of the $r$ random variables $\tau_{k_i+i-1}-\tau_{l_i+i}$ under $\mu_{N+r}$. The proof follows the same strategy as that of Proposition~\ref{prop:simple relation}, but relies on the general-$r$ version of the matrix product ansatz (Proposition~\ref{prop:MPA two}) introduced in~\cite{uchiyama2008two}. 

The main obstacle in extending our main theorems to general~$r$ through this formula lies in the subsequent asymptotic analysis: one is led to a sum of $2^r$ correlation functions at $r$ sites, each appearing with a sign $\pm1$. While \cite[equation~(3.19)]{uchiyama2005correlation} provides an exact formula for each correlation function $\mathbb{E}[\tau_{i_1}\cdots\tau_{i_r}]$ in terms of multiple contour integrals, carrying out the corresponding asymptotic analysis appears technically involved. 
For this reason, we restrict attention to the case $r=1$ in this section and leave the extension to general~$r$ for future work.
\end{remark}
}

The next result introduces the matrix product ansatz for the standard open ASEP.

\begin{proposition}[c.f. \cite{derrida1993exact}]\label{prop:MPA}
    Assume \eqref{eq:conditions open ASEP} and that there are matrices $\D$, $\E$, a row vector $\ll W|$ and a column vector $|V\rr$ with the same (possibly infinite) dimension, satisfying  
\be\label{eq:DEHP algebra} 
        \D\E-q\E\D=\D+\E, \quad
        \ll W|(\alpha\E-\gamma\D)=\ll W|, \quad
        (\beta\D-\delta\E)|V\rr=|V\rr 
 \ee
    (which is commonly referred to as the $\dehp$ algebra). Then we have 
    \be \label{eq:MPA open ASEP}
\mu_N\lb\tau_1,\dots,\tau_N\rb=\frac{\ll W|\prod_{i=1}^N\lb\one(\tau_i=0)\E+\one(\tau_i=1)\D\rb|V\rr}{\ll W|(\E+\D)^N|V\rr},
\ee 
for any $(\tau_1,\dots,\tau_N)\in\{0,1\}^N$, assuming that the denominator $Z_N:=\ll W|(\E+\D)^N|V\rr$ is nonzero.
\end{proposition}

\begin{remark} \label{rmk:denomenator 0} As noted by \cite{mallick1997finite,essler1996representations}, the denominator $Z_N:=\ll W|(\E+\D)^N|V\rr$ of the matrix product ansatz \eqref{eq:MPA open ASEP} being nonzero can be guaranteed by $ABCD\notin\{q^{-l}:l\in\NN_0\}$, where we recall that $A$, $B$, $C$ and $D$ are given by \eqref{eq:defining ABCD}. 
As shown in \cite[Appendix A]{mallick1997finite}, under the assumption that $\ll W|V\rr=1$, any matrix product of the form $\ll W|\D^{n_1}\E^{m_1}\dots\D^{n_k}\E^{m_k}|V\rr$ only depends on the parameters $q,\alpha,\beta,\gamma,\delta$ and numbers $n_1,m_1,\dots,n_k,m_k\in\ZZ_+$. In particular, it does not depend on the specific examples of matrices $\D$, $\E$ and vectors $\ll W|$, $|V\rr$ that satisfy the $\dehp$ algebra \eqref{eq:DEHP algebra}.  \end{remark} 

We will use the example constructed in \cite{uchiyama2004asymmetric}, which works for general parameters $q, \alpha, \beta, \gamma, \delta$ and is closely related to the Jacobi matrices for the Askey--Wilson orthogonal  polynomials \cite{askey1985some}.
We introduce this example as follows.
Assume $ABCD\notin\{q^{-l}:l\in\NN_0\}$.
For $m\in\NN_0$, we define $\alpha_m,\beta_m,\gamma_m,\delta_m,\ep_{m},\varphi_{m}$ in terms of $(A,B,C,D,q)$ by the formulas  in \cite[page 1243]{bryc2010askey}:
\be \label{eq:entries}
\begin{split}
    \alpha_m&=-AB q^m \beta_m ,\\
\beta_m&=\frac{1-A B C D q^{m-1} }
{\sqrt{1-q} (1-A B C D q^{2 m} ) (1-A B
C D q^{2 m-1} )} ,\\
\varepsilon_m&=\frac{(1-q^m)
(1-A C
q^{m-1} ) (1-AD
q^{m-1} ) (1-B C
q^{m-1} ) (1-B D q^{m-1} )
}{\sqrt{1-q} (1-A B C D q^{2
m-2} ) (1-A B C D q^{2 m-1} )} ,\\
\varphi_m&=-CD q^{m-1}\varepsilon_m ,\\
\gamma_m&= \frac{A}{\sqrt{1-q}}-\frac{\alpha_m}{A} (1-AC q^m)(1-ADq^m)-
\frac{A\varepsilon_m}{(1-AC q^{m-1})(1-ADq^{m-1})},\\
\delta_m&=\frac{1}{A\sqrt{1-q}}-\frac{\beta_m}{A} (1-AC q^m)(1-ADq^m)-
\frac{A\varphi_m}{(1-AC q^{m-1})(1-ADq^{m-1})}.
\end{split}
\ee 
 We note that $\gamma_m$ and $\delta_m$ above are well-defined by the choices of $\ep_m$ and $\varphi_m$.
These formulas are also well-defined for $q=0$ and/or $A=0$ by continuity.

Consider the infinite  tri-diagonal matrices:
\be\label{eq:D and E}\E=\frac{1}{1-q}\mathbf{I}+\frac{1}{\sqrt{1-q}}\mathbf{y},\quad \D=\frac{1}{1-q}\mathbf{I}+\frac{1}{\sqrt{1-q}}\mathbf{x},\ee
where $\bI$ denotes the infinite identity matrix,  
\be \label{eq:x and y}
      \mathbf{x}=\left[\begin{matrix}
        \gamma_0 & \ep_1 & 0 &\dots \\
        \alpha_0 &  \gamma_1& \ep_2 &\dots \\
        0 & \alpha_1 & \gamma_2& \dots\\
        \vdots &\vdots & \vdots& \ddots\\
      \end{matrix}\right], \quad \mathbf{y}=\left[\begin{matrix}
        \delta_0 & \varphi_1 & 0 &\dots \\
        \beta_0 & \delta_1 & \varphi_2 & \dots\\
        0 & \beta_1 & \delta_2 & \dots\\
        \vdots & \vdots& \vdots& \ddots\\
      \end{matrix}\right]\,
    \ee
and  infinite vectors 
\be\label{eq:W and V}\ll W|=[1,0,0,\dots],\quad |V\rr=[1,0,0,\dots]^T.\ee 
As shown in \cite[Section 2.1]{bryc2017asymmetric}, they satisfy  the $\dehp$ algebra \eqref{eq:DEHP algebra} with parameters $(q,\alpha,\beta,\gamma,\delta)$.

The next result introduces the matrix product ansatz for the open ASEP with a light particle.
\begin{proposition}[c.f. \cite{uchiyama2008two}]\label{prop:MPA two}
    Assume \eqref{eq:conditions open ASEP} and let $\D$, $\E$, $\ll W|$ and $|V\rr$ satisfy the $\dehp$ algebra~\eqref{eq:DEHP algebra}. Define $\A=\D\E-\E\D$.
     Then the stationary measure of open ASEP with a light particle can be written as 
    \be \label{eq:MPA open ASEP light}
    \mt_{N,1}\lb\tau_1,\dots,\tau_N\rb=\frac{\ll W|\prod_{i=1}^N\lb\one(\tau_i=0)\E+\one(\tau_i=1)\D+\one(\tau_i=\ha)\A\rb|V\rr}{[y]\ll W|(\E+\D+y\A)^N|V\rr},
    \ee 
for any $(\tau_1,\dots,\tau_N)\in\widehat{\Omega}_{N,1}$, under the assumption that the denominator \be\label{eq:denominator one light particle}\hZ_N:=[y]\big{\ll} W\big{|}(\E+\D+y\A)^N\big{|}V\big{\rr}\ee is nonzero. We mention that this notation denotes the coefficient of the monomial $y=y^1$ in the latter matrix product $\ll W|(\E+\D+y\A)^N|V\rr$, regarded as a polynomial with variable $y$.
\end{proposition}
\begin{remark}
    We note that the matrix product ansatz introduced in \cite{uchiyama2008two} is in a more general form than the version above. The matrix $\A$ only needs to satisfy the conditions $\A\E - q\E\A = \A$ and $\D\A - q\A\D = \A$. It can be verified that $\A = \D\E - \E\D$ satisfies these two relations, but the converse is in general not true. This matrix product ansatz can also provide the stationary measure for the open ASEP with $r$ light particles for general $r \in \mathbb{N}$, which we leave for future work.
\end{remark}

The issue of when the denominator $\hZ_N$ equals zero is not addressed in \cite{uchiyama2008two}. The following result provides a condition under which $\hZ_N$ is nonzero.
\begin{lemma}\label{lem:deno0}
    Assume $ABCD\notin\{q^{-l}:l\in\NN_0\}$ and $AC\neq 1$. Suppose $\D$, $\E$, $\ll W|$ and $|V\rr$ satisfy the DEHP algebra \eqref{eq:DEHP algebra}, $\ll W|V\rr=1$ and $\A=\D\E-\E\D$. Then $\hZ_N$ defined by \eqref{eq:denominator one light particle} is nonzero.
\end{lemma}
\begin{proof}
    In view of Remark \ref{rmk:denomenator 0}, the value of $\hZ_N$ is fixed under our assumptions. One can assume that the matrices $\D$, $\E$ and vectors $\ll W|$, $|V\rr$ are the ones provided by \eqref{eq:entries}, \eqref{eq:D and E}, \eqref{eq:x and y} and \eqref{eq:W and V}.
    
    We notice that the proof in \cite{uchiyama2008two} of the stationarity of the measure $\mt_{N,1}$ defined by  \eqref{eq:MPA open ASEP light} is by showing that $L^*f_N(\tau)=0$, where $L^*$ is the backwards generator of the system, and
    \[f_N(\tau):=\Bigg{\ll} W\Bigg{|}\prod_{i=1}^N\lb\one(\tau_i=0)\E+\one(\tau_i=1)\D+\one(\tau_i=\ha)\A\rb\Bigg{|}V\Bigg{\rr}\quad\mbox{for}\quad\tau=(\tau_1,\dots,\tau_N)\in\widehat{\Omega}_{N,1}.\]
    We note that \[\hZ_N=\sum_{\tau\in\widehat{\Omega}_{N,1}}f_N(\tau).\]
    Since the open ASEP with a light particle is an irreducible Markov process on state space $\widehat{\Omega}_{N,1}$, by the Perron-Frobenius theorem,   $\hZ_N=0$ implies that $f_N(\tau)$ is constantly $0$. We next proceed with a proof by contradiction and assume that $f_N(\tau)$ is constantly $0$.
    By linear combination we have
    \[\big{\ll} W\big{|}(\alpha\E-\gamma\D)^{N-1}\A\big{|}V\big{\rr}=0.\]
    Using    $\ll W|(\alpha\E-\gamma\D)=\ll W|$ multiple times and in view of $\A=\D\E-\E\D$, we have 
    $\ll W|\D\E|V\rr=\ll W|\E\D|V\rr$. Using \eqref{eq:D and E} we have $\ll W|\mathbf{x}\mathbf{y}|V\rr=\ll W|\mathbf{y}\mathbf{x}|V\rr$. Then in view of \eqref{eq:x and y} and \eqref{eq:W and V}, we have $\varepsilon_1\beta_0=\varphi_1\alpha_0$. Note that $\varphi_1=-CD\varepsilon_1$ and $\alpha_0=-AB\beta_0$ by \eqref{eq:entries}, and combining with $ABCD\neq1$  we get $\varepsilon_1\beta_0=0$. On the other hand, from the formulas \eqref{eq:entries} one can observe that $\beta_0\neq0$, and that $AC\neq1$ implies $\varepsilon_1\neq0$. Hence we have reached a contradiction.  The proof is concluded.
\end{proof}

We next provide the proof of Proposition \ref{prop:simple relation}.
 
\begin{proof}[Proof of Proposition \ref{prop:simple relation}]
Recall our assumptions $ABCD\notin\{q^{-l}:l\in\NN_0\}$ and $AC\neq1$. 
In view of Remark \ref{rmk:denomenator 0} and Lemma \ref{lem:deno0} respectively, we have $Z_{N+1}\neq0$ and $\hZ_N\neq0$.

By Proposition \ref{prop:MPA} we have, for $i \in \lBr N \rBr$,
    \[\m_{N+1}(\tau_i=1)-\m_{N+1}(\tau_{i+1}=1)=\frac{1}{Z_{N+1}}\big{\ll} W\big{|}(\D+\E)^{i-1}(\D\E-\E\D)(\D+\E)^{N-i}\big{|}V\big{\rr}.\]
     By Proposition \ref{prop:MPA two} we have, for $1\leq i\leq N$,
    \[\mt_{N,1}\lb\lc_1=i\rb=\frac{1}{\hZ_N}\big{\ll} W\big{|}(\D+\E)^{i-1}\A(\D+\E)^{N-i}\big{|}V\big{\rr}.\]
Comparing the above two equations, we have 
\[\m_{N+1}(\tau_i=1)-\m_{N+1}(\tau_{i+1}=1)=\frac{\hZ_N}{Z_{N+1}}\mt_{N,1}\lb\lc_1=i\rb.\]
For any $1\leq k\leq l\leq N$, by adding up the above equation for $i \in \lBr k,l \rBr$, we have 
    \be \label{eq:relation open ASEP to light}
    \m_{N+1}(\tau_k=1)-\m_{N+1}(\tau_{l+1}=1)
    =\frac{\hZ_N}{Z_{N+1}} \mt_{N,1}\lb k\leq\lc_1\leq l\rb.
    \ee 
    In particular, taking $k=1$ and $l=N$, we have 
    \be \label{eq:normalization ratio}
    \m_{N+1}(\tau_1=1)-\m_{N+1}(\tau_{N+1}=1)=\frac{\hZ_N}{Z_{N+1}}.
    \ee 
Comparing \eqref{eq:relation open ASEP to light} and \eqref{eq:normalization ratio}, we have 
\[\mt_{N,1}\lb k\leq\lc_1\leq l\rb=\frac{\m_{N+1}(\tau_k=1)-\m_{N+1}(\tau_{l+1}=1)}{\m_{N+1}(\tau_1=1)-\m_{N+1}(\tau_{N+1}=1)}.\]
This finishes the proof.
\end{proof}  
 
\section{Densities at the open boundaries}\label{sec:proof of boundary density}
In this section, we first use the Askey--Wilson polynomials to prove the formulas for the limiting densities at the left and right boundaries of open ASEP, as stated in Proposition \ref{prop:boundary density}. Combined with the simple relation (Proposition \ref{prop:simple relation}) obtained in the previous section, we give the proof of Theorem \ref{thm:main thm} concerning the maximal current phase and also the proof of Theorem \ref{thm:mainHighLow} concerning the high and low density phases, in the case when $r=1$.

\subsection{Proof of Proposition \ref{prop:boundary density}}\label{subsec: 3.1}
We first review some background material. Consider the open ASEP stationary measure on the lattice $\lBr N \rBr$ with fixed jump rates $q,\alpha,\beta,\gamma$ and $\delta$. It is known since \cite{liggett1975ergodic} that, as $N\rightarrow\infty$, this sequence of measures weakly converges to a probability measure on $\{0,1\}^{\ZZ_+}$, and that the limiting measure can be seen as a (non-unique) stationary measure of a certain ASEP system on the semi-infinite lattice $\ZZ_+$. See also \cite{liggett1975ergodic,grosskinsky2004phase,sasamoto2012combinatorics,bryc2017asymmetric,yang2024limits} for studies of this limit. By our definition, the marginal distribution of this limiting measure at the first site is simply the Bernoulli measure with probability $\ld$. 

As mentioned in Section \ref{subsec:related works}, the stationary measure of the open ASEP can be effectively characterized by techniques involving Askey--Wilson polynomials and Askey--Wilson signed measures, as developed by \cite{uchiyama2004asymmetric,bryc2010askey,bryc2017asymmetric,wang2023askey}. The limiting measure on $\{0,1\}^{\ZZ_+}$ mentioned in the previous paragraph has also been characterized by Askey--Wilson signed measures in \cite[Section 5]{bryc2017asymmetric} in the fan region and in \cite{yang2024limits} in the shock region. We have the following result.
 
\begin{lemma}[c.f. \cite{bryc2017asymmetric,yang2024limits}] \label{lem:characterization of ld}
    We have the following characterizations of $\ld$.
    \begin{enumerate}
        \item [$\bullet$] (low density phase) Assume $C>1$ and $C>A$, we have $\ld=1/(1+C)$.
        \item [$\bullet$] (high density phase) Assume $A>1$ and $A>C$. Assume also that $C/A\notin\{q^l:l\in\ZZ_+\}$ if $C\geq1$. Then there exists some $\ep>0$ such that for $t\in(1-\ep,1]$,
    \be\label{eq:in proof HD}t\ld+1-\ld=\frac{A}{(1+A)^2}\int_{\RR}\lb1+t+2\sqrt{t}x\rb\nu\lb\mathsf{d} x; A\sqrt{t},\frac{\sqrt{t}}{A},\frac{C}{\sqrt{t}},\frac{D}{\sqrt{t}}\rb.\ee
        \item [$\bullet$] (maximal current phase) Assume $A<1$ and $C<1$. Then for $t\in(0,1]$,
    \be\label{eq:in proof MC} t\ld+1-\ld=\frac{1}{4}\int_{\RR}\lb1+t+2\sqrt{t}x\rb\nu\lb\mathsf{d} x; \sqrt{t},\sqrt{t},\frac{C}{\sqrt{t}},\frac{D}{\sqrt{t}}\rb.\ee
    \end{enumerate}
    
    We mention that $\nu\lb\mathsf{d} x;a,b,c,d\rb$ is the Askey--Wilson signed measure introduced in \cite{wang2023askey}. They are finite signed measures compactly supported in $\RR$ and with total mass $1$, which depend on parameters $a,b,c,d$ and $q$ in a suitable parameter region. Since we will not use the explicit form of $\nu\lb\mathsf{d} x;a,b,c,d\rb$, we refer the reader to Definitions 2.1 and 2.2 in \cite{wang2023askey} for a detailed definition.   
\end{lemma}
\begin{proof}
    We reiterate that the Bernoulli measure with probability $\ld$ is the marginal distribution at the first site of the limiting measure (on $\{0,1\}^{\ZZ_+}$) for the open ASEP on $\lBr N \rBr$ as $N\rightarrow\infty$. 
    
    In the low density phase, by \cite[Theorem 1.2]{yang2024limits}, the limiting measure is simply the Bernoulli product measure on $\{0,1\}^{\ZZ_+}$ with probability $1/(1+C)$, hence $\ld=1/(1+C)$.  

    In the high density phase, the limiting measure is not a Bernoulli product measure anymore, and it is characterized by \cite[Theorem 1.2]{yang2024limits} in terms of Askey--Wilson signed measures. Equation \eqref{eq:in proof HD} follows from taking $m=1$ therein.

    In the maximal current phase, \eqref{eq:in proof MC} follows from taking $K=1$ in \cite[Theorem 12]{bryc2017asymmetric}.
\end{proof}
 
Next we need to compute the right-hand sides of \eqref{eq:in proof HD} and \eqref{eq:in proof MC} to obtain explicit formulas for $\ld$.
We establish the following result regarding the first moments of Askey--Wilson signed measures.
    \begin{lemma}\label{lem:AW moment}
    Suppose that the Askey--Wilson signed measure $\nu\lb\dd x;a,b,c,d\rb$ is well-defined, i.e., $(a,b,c,d)\in\Omega_q$, where $\Omega_q$ is the subset of $\mathbb{C}^4$ defined in \cite[Definition 2.1]{wang2023askey}. Then we have 
        \be\label{eq:first moment}\int_{\RR}x\nu\lb\dd x;a,b,c,d\rb=\frac{a+b+c+d-abc-abd-acd-bcd}{2(1-abcd)}.\ee
    \end{lemma}
    \begin{proof} 
    We recall from \cite[Corollary 2.9]{wang2023askey} that Askey--Wilson signed measures integrate to $1$,  
    \be\label{eq:integrate 1}\int_{\RR}\nu\lb\dd x;a,b,c,d\rb=1.\ee
    By the orthogonality of the Askey--Wilson signed measures in \cite[Corollary 3.2]{wang2023askey},
    \be\label{eq:ort}\int_\mathbb{R}w_1(x)\nu(\mathsf{d} x;a,b,c,d)=0,\ee
    where $w_1(x)$ is the degree-$1$ Askey--Wilson polynomial. 
    By explicitly computing the linear polynomial $\omega_1(x)$ using the recurrence relations from \cite{askey1985some} (see also \cite[Section 2.1]{wang2023askey}), equation \eqref{eq:first moment} follows from \eqref{eq:integrate 1} and \eqref{eq:ort}. Note that this computation is a bit lengthy but actually does not need to be carried out. In fact, under $\max\lb|a|,|b|,|c|,|d|\rb<1$, equation \eqref{eq:first moment} appears as \cite[equation (2.6)]{bryc2010askey}. By \eqref{eq:integrate 1} and \eqref{eq:ort}, one can observe that the left-hand side of \eqref{eq:first moment} is a rational function of $a,b,c$ and $d$. Hence \eqref{eq:first moment} holds for general $(a,b,c,d)\in\Omega_q$. We conclude the proof.
    \end{proof}

Before we give the proof of Proposition \ref{prop:boundary density}, we recall the ``stochastic sandwiching'' of the open ASEP stationary measure. For two probability measures $\nu,\nu^{\prime}$ on $\Omega_N$, we say that $\nu$ \textbf{stochastically dominates} $\nu^{\prime}$, and write $\nu \succeq \nu^{\prime}$, whenever there exists a coupling $\mathbf{P}$ between $\nu$ and $\nu^{\prime}$ such that
\begin{equation}\label{eq:stochastic domination}
    \mathbf{P}( \eta_i \geq \eta^{\prime}_i \text{ for all } i \in \lBr N \rBr ) = 1 , 
\end{equation}  where $\eta \sim \nu$ and $\eta^{\prime} \sim \nu^{\prime}$.
\begin{lemma}[c.f. Lemma 2.1 in \cite{gantert2023mixing}]\label{lem:sandwiching}
    Fix $q\in[0,1)$ and consider real numbers
   \begin{equation}\label{eq:Sandwich}
       0<\alpha' \leq\alpha'',\quad\beta' \geq\beta''>0,\quad\gamma' \geq\gamma''\geq0,\quad0\leq\delta' \leq\delta''.
   \end{equation} 
Consider the open ASEP on the lattice $\lBr N \rBr$ with rates $(q,\alpha',\beta',\gamma',\delta')$ and $(q,\alpha'',\beta'',\gamma'',\delta'')$. We denote the stationary measures as $\mu_N'$ and $\mu_N''$, respectively. Then we have $\mu_N''\succeq\mu_N'$.
\end{lemma}

\begin{proof}[Proof of Proposition \ref{prop:boundary density}]
{
By the particle--hole symmetry of the open ASEP, the stationary measures $\mu_N^{(A,B,C,D)}$ and $\mu_N^{(C,D,A,B)}$ are related by the transformation of occupation variables
\begin{equation}\label{eq:particle hole}
    \widehat{\tau}_i := 1 - \tau_{N+1-i}, \qquad i=1,\dots,N,
\end{equation}
see, for example,~\cite[Lemma~2.7]{yang2024limits} for the statement. 
This follows from the fact that if one treats particles as holes and holes as particles in an open ASEP with parameters $(q,\alpha,\beta,\gamma,\delta)$, then—after relabeling the sites $\{1,\dots,N\}$ as $\{N,\dots,1\}$—the resulting dynamics correspond to an open ASEP with parameters $(q,\beta,\alpha,\delta,\gamma)$; see, for example,~\cite[Section~4.2]{wang2023askey} for a detailed discussion. 
Taking $i=1$ and letting $N\to\infty$ in~\eqref{eq:particle hole}, we obtain
\[
\rd(C,D,A,B) = 1 - \ld(A,B,C,D).
\]
Therefore,} we only need to prove the formulas for $\ld$. Using Lemma \ref{lem:AW moment}, one can explicitly compute the right-hand sides of \eqref{eq:in proof HD} and \eqref{eq:in proof MC} and extract the formulas for $\ld$, which match the formulas provided in Proposition \ref{prop:boundary density}.
By Lemma \ref{lem:characterization of ld}, the result for $\ld$ is thus proved everywhere except for countably many ``exceptional curves'': the boundaries between phases $\{(A,C):A\leq1,C=1\}$, $\{(A,C):A=1,C\leq1\}$, $\{(A,C):A=C\geq1\}$ as well as $\{(A,C):C/A=q^l, A,C\geq1\}$ for $l\in\ZZ_+$. 

We recall from Remark \ref{rmk:left and right density} that the formulas for $\ld$ given in Proposition \ref{prop:boundary density} define a continuous function of $A,B,C$ and $D$ in the entire parameter space. In the following we will use this continuity and ``stochastic sandwiching'' to show that the result also holds along those ``exceptional curves''. 
 

Assume that $(A,C)$ belongs to one of the ``exceptional curves''. Then we can find sequences $\alpha'_k\rightarrow\alpha$,  $\alpha''_k\rightarrow\alpha$, $\beta'_k\rightarrow\beta$ and $\beta''_k\rightarrow\beta$ such that for each $k \in \N$ we have $\alpha'_k\leq\alpha\leq\alpha''_k$ and $\beta'_k \geq\beta\geq\beta''_k$; and that the corresponding $(A'_k,C'_k)$ and $(A''_k,C''_k)$ defined by $A'_k:=\phi_+\lb\beta'_k,\delta\rb$, $A''_k:=\phi_+\lb\beta''_k,\delta\rb$, $C'_k:=\phi_+\lb\alpha'_k,\gamma\rb$ and $C''_k:=\phi_+\lb\alpha''_k,\gamma\rb$ do not lie on these curves, where we recall  $\phi_{\pm}$ are defined by \eqref{eq:phi}. Using Lemma \ref{lem:sandwiching} above and taking the limit $N\rightarrow\infty$, we have $\ld\lb A'_k,B,C'_k,D\rb\leq\ld\lb A,B,C,D\rb\leq\ld\lb A''_k,B,C''_k,D\rb$ for each $k\in \N$. Since Proposition \ref{prop:boundary density} has already been proved at $(A'_k,B,C'_k,D)$ and $(A''_k,B,C''_k,D)$, by taking the limit $k\rightarrow\infty$ in the inequality above, we conclude the proof of the result at $(A,B,C,D)$.
\end{proof}

\subsection{Proof of the main theorems in the case of a single light particle}\label{subsec: 3.2} 

We have now all the tools to show our main theorems regarding the location of the light particle for $r=1$.

\begin{proof}[Proof of Theorem \ref{thm:main thm} ]
Recall our assumptions $\max(A,C) \leq 1$ and $AC \neq 1$. 
Let $(a_N)_{N=1,2,\dots}$ and $(b_N)_{N=1,2,\dots}$ satisfy  $1\leq a_N<b_N\leq N$ and $a_N,N-b_N\ra\infty$.
By Proposition \ref{prop:simple relation},
\be\label{eq:simple relation again}
\mt_{N,1}\lb k\leq\lc_1\leq l\rb=\frac{\m_{N+1}(\tau_k=1)-\m_{N+1}(\tau_{l+1}=1)}{\m_{N+1}(\tau_1=1)-\m_{N+1}(\tau_{N+1}=1)}.
\ee
By Proposition \ref{prop:bulk density},
\[\lim_{N\ra\infty}\m_{N+1}(\tau_{a_N+1}=1)=\frac{1}{2}.\]
By Proposition \ref{prop:boundary density},
\begin{align*}
    &\lim_{N\ra\infty}\m_{N+1}(\tau_1=1)=\ld=\frac{3-C-D-CD}{4(1-CD)}>\frac{1}{2},\\
    &\lim_{N\ra\infty}\m_{N+1}(\tau_{N+1}=1)=\rd=\frac{1 +A+B-3AB}{4(1-AB)}<\frac{1}{2}.
\end{align*} 
Taking $k=1$ and $l=a_N$ in \eqref{eq:simple relation again}, we have 
\[\lim_{N\ra\infty}\mt_{N,1}\lb\lc_1\in \lBr 1,a_N\rBr \rb=\frac{\ld-1/2}{\ld-\rd}. \]
Similarly, taking $k=b_N$ and $l=N$ in \eqref{eq:simple relation again} we have
\[\lim_{N\ra\infty}\mt_{N,1}\lb\lc_1\in \lBr b_N,N\rBr \rb=\frac{1/2-\rd}{\ld-\rd}.\]
The fact that 
\[\lim_{N\ra\infty}\mt_{N,1}\lb\lc_1\in \lBr a_N,b_N\rBr \rb=0\]
is a simple corollary of the above two limits. We conclude the proof.
\end{proof}

\begin{proof}[Proof of Theorem \ref{thm:mainHighLow} for $r=1$, $ABCD\notin\{q^{-l}:l\in\NN_0\}$ and $AC\neq1$]
We will only prove the result in the high density phase $A\geq 1$ and $A>C$. 
The result in the low density phase can be proved symmetrically.
By Proposition \ref{prop:bulk density},
\[\lim_{N\ra\infty}\m_{N+1}(\tau_{a_N+1}=1)=\frac{A}{1+A}.\]
By Proposition \ref{prop:boundary density},
\begin{align*}
    &\lim_{N\ra\infty}\m_{N+1}(\tau_1=1)=\ld=\frac{A^2+A+1-ACD-AC-AD}{(A+1)^2(1-CD)},\\ 
    &\lim_{N\ra\infty}\m_{N+1}(\tau_{N+1}=1)=\rd=\frac{A}{1+A}.
\end{align*} 
We recall from Remark \ref{rmk:left and right density} that $\ld\neq\rd$ since $AC\neq1$.
Taking $k=1$ and $l=a_N$ in \eqref{eq:simple relation again}, we have 
\[\lim_{N\ra\infty}\mt_{N,1}\lb\lc_1\in \lBr 1,a_N\rBr \rb=1. \] 
The proof is concluded.
\end{proof}

\section{Concentration of light particles in the high and the low density phases}\label{sec:Concentration}

In this section we prove Theorem \ref{thm:mainHighLow} and Theorem \ref{thm:MainMixingTime} in the fan region of the high and low density phases. 
We will always assume $A>1$ and $AC\leq 1$, i.e.\ we consider the open ASEP in the fan region of the high density phase with a finite number $r$ of light particles at locations $\lc_1(t),\lc_2(t),\dots,\lc_r(t)$ at time $t \geq 0$. Using the particle-hole symmetry, the results also hold when $C>1$ and $AC\leq 1$, i.e.\ the fan region of the low density phase. 

Let us start with an outline of the arguments.

\subsection{The strategy for the proof}\label{subsec:strategy for the proof}

We have the following strategy in order to show Theorem \ref{thm:mainHighLow} on the locations of the light particles for $r \geq 1$, and Theorem \ref{thm:MainMixingTime} on the mixing time. 

In Section \ref{sec:SecondClassProjection}, we introduce a color projection of the open ASEP with light particles.  This will allows us to reduce the study of the location of the light particles under the stationary measure to the case when we choose the configuration for the remaining sites according to the stationary measure of a standard open ASEP. Moreover, we discuss stochastic domination for the stationary measure of the open ASEP with different boundary parameters.  

In Section \ref{sec:ASEPintegers}, we present the asymmetric simple exclusion process (ASEP) on the integer lattice $\Z$ with second class particles, and discuss the basic coupling to the open ASEP with light particles. This allows us to compare exclusion processes with different locations of second class particles. 

In Section \ref{sec:Moderate}, we discuss recent moderate deviation results for second class particles in the ASEP on the integers by Landon and Sosoe in \cite{landon2023tail} started from stationary initial data. We apply them to the open ASEP in the fan region using the so-called microscopic concavity coupling by Balazs and Seppäläinen from \cite{balazs2009fluctuation}.


{In Section~\ref{sec:EscapeLight}, we consider an open ASEP with light particles initially placed at arbitrary but fixed locations. Assuming that the configuration outside these fixed locations is distributed according to the stationary measure of the open ASEP without light particles, we argue that} it takes time of order at most $N^{\theta}\log(N)$ for some $\theta>0$ until the light particles are with high probability at distance at least $N^{\theta}$ away from the right boundary. This is achieved by a multi-scale argument together with the moderate deviation results from Section~\ref{sec:Moderate}. We then apply a similar multi-scale argument to ensure that when starting from a typical position of the second class particles in the high density phase, at any sufficiently large time $t$, there is a positive probability for the particles to be close to the left boundary. This allows us to conclude Theorem \ref{thm:mainHighLow} in the fan region. 
 
In Section \ref{sec:MixingTimesLight}, we establish Theorem \ref{thm:MainMixingTime} using  estimates in previous subsections. 
More precisely, we first argue that the color projection to an open ASEP without light particles mixes in time of order~$N$ using the results by Gantert et al.\ in \cite{gantert2023mixing}. We then obtain a result on the coalescence time, which implies that the light particles have mixed after an additional order $N$ steps.

{ 
\begin{remark}
It is natural to ask whether the approach developed in this section, which relies on a detailed analysis of the dynamics of the $r$ light particles in the open ASEP, can be extended to other parts of the phase diagram. 
For the shock region of the high- and low-density phases, we expect that such an extension could apply given an additional technical input; see Remark~\ref{rem:extend to shock region}. 
By contrast, in the maximal current phase and along the coexistence line, we currently lack a clear conjectural description of the dynamics of multiple light particles in the open ASEP. 
This prevents an extension of Theorems~\ref{thm:main thm} and~\ref{thm:mainCoexistence} to general~$r$ using methods similar to those in this section.
\end{remark} 
}

\subsection{A color projection for the open ASEP}\label{sec:SecondClassProjection}

We consider a projection of the open ASEP with light particles, where we replace the light particle by either a first class particle or an empty site. More precisely, let $(\textup{Col}(t))_{t \geq 0}$ be a \textbf{coloring process} for the light particle, i.e., for each $t\geq 0$, let  $\textup{Col}(t) \in \{0,1\}^r$ such that $t \mapsto \textup{Col}(t)$ is càdlàg. Consider the open ASEP with light particles $(\tau(t))_{t \geq 0}$ and a coloring process $(\textup{Col}(t))_{t \geq 0}$. Then its \textbf{color projection} $(\eta(t))_{t \geq 0}$ is defined as 
\begin{equation}
\eta_x(t) = \begin{cases} \textup{Col}_i(t) & \text{ if } x = \lc_i(t) \text{ for some } i \in \lBr r \rBr \\
\tau_x(t) & \text{ otherwise. } 
\end{cases}
\end{equation} for all $t \geq 0$ and $x\in \lBr N \rBr$. We make the following observation.
\begin{lemma}\label{lem:LightProjection}
There exists a coloring process $(\textup{Col}(t))_{t \geq 0}$ such that the color projection 
$(\eta(t))_{t \geq 0}$ has the law of an open ASEP with parameters $q,\alpha,\beta,\gamma,\delta$, but without light particles.
\end{lemma}
\begin{proof}
Let the initial coloring $\textup{Col}(0) \in \{0,1\}^r$ be chosen according to an arbitrary rule. Then let $\textup{Col}(t)$ remain constant apart from the following exceptions. If $\lc_{i+1}(t)=\lc_i(t)+1$ for some {$i\in  \lBr r-1 \rBr$,} then at rate $1$, we swap the values $\textup{Col}_i(t)$ and $\textup{Col}_{i+1}(t)$  if $\textup{Col}_i(t)>\textup{Col}_{i+1}(t)$, and we swap them at rate $q$, otherwise.
Whenever $\lc_1(t) = 1$, then at rate $\alpha$, we replace the current value of $\textup{Col}_1(t)$ by $1$, and at rate $\gamma$ by $0$. Similarly, when $\lc_r(t)=N$, we replace $\textup{Col}_r(t)$ by $0$ at rate $\beta$, and by  $1$ at rate $\delta$. The statement is now immediate by verifying the transition rates. 
\end{proof}

\begin{remark}\label{rem:ProjectionOpenASEP}
From Lemma \ref{lem:LightProjection}, we see that up to the location of the light particles, the stationary measure $\widehat{\mu}_{N,r}$ agrees with the stationary measure $\mu_N$ of an open ASEP without light particles. More precisely, consider the probability measure $\widehat{\mu}_{N,r}$ on the space $\widehat{\Omega}_{N,r}$ of configurations with $r$ many light particles. Then for each configuration $\tau \in \widehat{\Omega}_{N,r}$, one can project its $r$ light particles to either $0$ or $1$, according to a certain law, such that the resulting measure on the space $\{0,1\}^{N}$ is exactly the stationary measure $\mu_N$ of the open ASEP without light particles.
\end{remark}

Next, we record an observation on the stationary measure of the open ASEP {with general parameters.} 
The following statement can be found for example as \cite[Lemma 2.10]{gantert2023mixing}, and follows from Lemma \ref{lem:sandwiching}, as $\mu_N$ is a Bernoulli-$1/(1+C)$-product measure when $AC=1$.
Recall that $\succeq$ denotes stochastic domination which is defined in \eqref{eq:stochastic domination}.
\begin{lemma}\label{lem:StochasticDomination}
    The stationary measure $\mu_N$ of open ASEP on $\llbracket N\rrbracket$ satisfies 
\begin{equation}
    \textup{Bern}_{N}(c_{\max}) \succeq \mu_N \succeq     \textup{Bern}_{N}(c_{\min}) , 
\end{equation} where $c_{\max}:=\max\left(1/(1+C),A/(1+A)\right)$ and $c_{\min}:=\min\left(1/(1+C),A/(1+A)\right)$, and where $\textup{Bern}_{N}(\rho)$ denotes the Bernoulli-$\rho$-product measure on $\{0,1\}^{N}$ for some $\rho \in [0,1]$. \end{lemma}

\subsection{The ASEP on the integers}\label{sec:ASEPintegers}

In order to study the open ASEP with light particles, it will be convenient to consider the asymmetric simple exclusion process on the integer lattice with second class particles. This is a Markov process $(\zeta(t))_{t \geq 0}$ on $\{0,1,2 \}^{\Z}$ where we order the spins according to the {total} ordering $1 \succeq 2 \succeq 0$.
We refer to a site $x$ be occupied by a \textbf{(first class) particle} at time $t$ if $\zeta_x(t)=1$, by a \textbf{second class particle} if $\zeta_x(t)=2$, and having an \textbf{empty site}, otherwise. In words, first class particles receive the highest priority, then second class particles, and then empty sites. In particular, note that the second class particles have the same transition rules as the light particles in the bulk of the open ASEP with light particles. 

In order to define the ASEP on the integers for different initial conditions simultaneously, we consider the \textbf{basic coupling},  denoted by $\mathbf{P}$.  For each edge $e=\{x,x+1\}$ with $x\in \Z$, we assign independent rate $1$ and rate $q$ Poisson clocks. Whenever the rate $1$ clock rings, we sort the spins at the endpoints along $e$ in increasing order. When the rate $q$ clock rings, we sort them in decreasing order.  The basic coupling has the advantage that we can construct the ASEP with second class particles by taking the component-wise difference between two ASEPs without second class particles. This is in contrast to the open ASEP with light particles, where the light particles are not allowed to leave the segment. More precisely, consider a pair of ordered initial configurations $\zeta^{1}$ and $\zeta^{2}$, i.e.
\begin{equation}
    \zeta^{1}_x \geq \zeta^{2}_x
\end{equation} for all $x\in \mathbb{Z}$. 
We define the \textbf{disagreement process} $(\xi(t))_{t \geq 0}$ by the relation
\begin{equation}\label{def:Disagreement}
    \xi_x(t) = \zeta^1_x(t)\mathds{1}_{\zeta^1_x(t) = \zeta^2_x(t)} + 2 \mathds{1}_{\zeta^1_x(t) \neq  \zeta^2_x(t)}
\end{equation}
for two exclusion processes $(\zeta^{1}(t))_{t \geq 0}$ and $(\zeta^{2}(t))_{t \geq 0}$ according to the basic coupling, started from $\zeta^{1}$ and $\zeta^{2}$, respectively. Note that $(\xi(t))_{t \geq 0}$ has indeed the law of an ASEP on the integers with second class particles. The next lemma shows that this construction can be extended to a coupling between the ASEP on the integers with second class particles and the open ASEP with light particles.   

\begin{lemma}\label{lem:CouplingOpenZ}
Consider an interval $\lBr a,b \rBr \subset \lBr N \rBr$, and fix initial configurations $\tau$ and $\zeta$ for open ASEP with light particles $(\tau(t))_{t \geq 0}$ and ASEP on the integers with second class particles $(\zeta(t))_{t \geq 0}$, respectively. We assume that the configurations $\tau$ and $\zeta$ agree on the sites $\lBr a, b \rBr$, i.e.,
$\zeta_x=\tau_x$
for all  $x\in \lBr a, b \rBr$. Then there exists a coupling $\tilde{\mathbf{P}}$ and constants $c_0,C_0>0$ such that for all $t \geq 0$
\begin{equation}\label{eq:baer}
\tilde{\mathbf{P}}\left( \tau_x(s) = \zeta_x(s) \text{ for all } s \in [0,t] \text{ and } x \in [a+2t,b-2t] \right) \geq 1- C_0 \exp(-c_0 t ) . 
\end{equation}
\end{lemma}
\begin{proof} For the coupling $\tilde{\mathbf{P}}$, we use the same rate $1$, respectively rate $q$, Poisson clocks between pairs of sites $\{x,x+1\}$ for some $x\in \lBr N-1 \rBr$ in both processes $(\tau(t))_{t \geq 0}$ and $(\zeta(t))_{t \geq 0}$. Along all other edges on $\Z$ for $(\zeta(t))_{t \geq 0}$, as well as for the entering and exiting of particles in the process $(\tau(t))_{t \geq 0}$, we use independent Poisson clocks with corresponding rates.
 Note that  following the same rules as under the basic coupling, 
 the two processes can only disagree on interval $[a+2t,b-2t]$ at some time in $[0,t]$ when there exist times $0 \leq t_1\leq \dots \leq t_m$ for $m= \lfloor 2t \rfloor$ such that for all $x\in \lBr m \rBr$, either the Poisson clock on $\{a+x-1,a+x\}$ or on $\{b-x,b-x+1\}$ rings at time $t_x$. 
 Therefore the probability that they disagree on $[a+2t,b-2t]$ can be bounded by the probability of $Z_1+\dots+Z_m<t$, where $Z_i$ for $i\in\llbracket m\rrbracket$ are i.i.d. exponential random variables with parameter $1+q<2$.
 By the standard large deviation principle of the exponential distribution, we conclude the proof of \eqref{eq:baer}. 
\end{proof}
\begin{remark}
    With a slight abuse of notation, we will refer to the coupling constructed in Lemma~\ref{lem:CouplingOpenZ} above also as the basic coupling. 
\end{remark}
 
 Next, we aim to use the basic coupling to compare the positions of (finitely many) second class particles in two ASEPs on the integers. More precisely, we consider a pair of disagreement processes $(\xi(t))_{t \geq 0}$ and $(\xi^{\prime}(t))_{t \geq 0}$, where the four underlying exclusion processes on $\Z$ are all coupled according to the { basic coupling $\mathbf{P}$. Here, we recall that in the basic coupling, for each edge $e=\{x,x+1\}$ with $x\in \Z$, we assign independent rate $1$ and rate $q$ Poisson clocks. Whenever the rate $1$ clock rings, we sort the spins at the endpoints along $e$ in increasing order in all four exclusion processes. When the rate $q$ clock rings, we sort them in decreasing order. Under this basic coupling of the four exclusion processes, the processes $(\xi(t))_{t \ge 0}$ and $(\xi'(t))_{t \ge 0}$ arise as disagreement processes associated with two pairs of these exclusion processes, as defined in~\eqref{def:Disagreement}.} The initial conditions $\xi$ and $\xi^{\prime}$ of the two disagreement processes are chosen as follows.
Let $S,S^{\prime} \subseteq \Z$ with $|S|,|S^{\prime}| < \infty$, and let $\zeta \in \{0,1\}^{\mathbb{Z}}$ be chosen in an arbitrary way. Then we set
\begin{equation}\label{eq:condi1}
\xi_x=\begin{cases} \zeta_x & \text{ if } x\notin S \\
2 & \text{ if } x\in S
\end{cases}
\end{equation} as well as
\begin{equation}\label{eq:condi2}
\xi^{\prime}_x=\begin{cases} \zeta_x & \text{ if } x\notin S^{\prime} \\
2 & \text{ if } x\in S^{\prime} . 
\end{cases}
\end{equation}
We enumerate the second class particles from left to right in $(\xi(t))_{t \geq 0}$ by $(Z^{i,1}(t))_{t \geq 0}$ for $i\in \lBr |S|\rBr$, and by   $(Z^{i,2}(t) )_{t \geq 0}$ for $i\in \lBr |S^{\prime}| \rBr$ in $(\xi^{\prime}(t))_{t \geq 0}$. Furthermore, we say that a second class particle $(Z^{i,j}(t) )$ is of type $0,1,2$ at time $t$, and write $\ell(Z^{i,j}(t) )\in\{0,1,2\}$, depending on whether it is matched with an empty site, first, or second class particle under the basic coupling, { i.e., we set
\begin{equation}
   \ell(Z^{i,1}(t)) = \xi^{\prime}_{Z^{i,1}(t)}(t) \quad \text{ and } \quad \ell(Z^{i,2}(t)) = \xi_{Z^{i,2}(t)}(t)
\end{equation} for all $i\in \lBr |S|\rBr$ and $i^{\prime}\in \lBr |S^{\prime}|\rBr$, and $t\geq 0$.}
 Note that while the number of second class particles in both disagreement processes is preserved, their types can vary over time. We provide the following result which compares the position of the second class particles in $(\xi(t))_{t \geq 0}$ and $(\xi^{\prime}(t))_{t \geq 0}$, started from $\xi$ and $\xi^{\prime}$, respectively. 

\begin{lemma}\label{lem:CouplingPositions}
Assume that $s^{\prime}>s$ for all $s \in S$ and $s^{\prime} \in S^{\prime}$. Moreover, let 
\begin{equation}\label{eq:DominationAssumption}
\begin{split}
 \left| \left\{  i \in \lBr |S| \rBr \, \colon \, \ell\left(Z^{i,1}(0) \right)=1 \right\} \right| &\leq   \left| \left\{  i \in \lBr |S^{\prime}| \rBr \, \colon \, \ell\left(Z^{i,2}(0) \right)=1 \right\} \right| \\
 \left| \left\{  i \in \lBr |S| \rBr \, \colon \, \ell\left(Z^{i,1}(0) \right)=0 \right\} \right| &\leq   \left| \left\{  i \in \lBr |S^{\prime}| \rBr \, \colon \, \ell\left(Z^{i,2}(0) \right)=0 \right\} \right|
\end{split}
\end{equation}
Then for all $t \geq 0$, we have that almost surely
\begin{equation}\label{eq:DominationStatement}
 Z^{|S|,1}(t)  \leq   Z^{|S^{\prime}|,2}(t)  . 
\end{equation}
Similarly, suppose that $s > s^{\prime}$ for all $s \in S$ and $s^{\prime} \in S^{\prime}$ as well as \eqref{eq:DominationAssumption} holds.
Then for all $t \geq 0$, we have that almost surely
\begin{equation}\label{eq:Dom2}
 Z^{1,2}(t)  \leq   Z^{1,1}(t)  . 
\end{equation}
\end{lemma}
\begin{proof}
Under the basic coupling, we first describe the dynamics of the particle pairs. We write $(i,j)$ to mean that $\xi_x(t)=i$ and $\xi^{\prime}_x(t)=j$ for some fixed $x\in \Z$ and $t \geq 0$. Observe that when starting from \eqref{eq:condi1} and \eqref{eq:condi2}, the particle pairs $(1,0)$ and $(0,1)$ can never occur. The possible pairs are
\begin{equation}\label{eq:AllConfigurationsPossible}
 (1,1) , \  (2,2) , \ (0,0) , \ 
(2,1) , \  (2,0) , \  (1,2) , \ (0,2) .
\end{equation}
We construct a natural partial ordering $\succeq_{o}$ of particle pairs, see Figure \ref{fig:SecondClassHierachy} for an illustration. Note that any two pairs can be compared, except for the pairs $(2,1)$ and $(1,2)$, as well as $(2,0)$ and $(0,2)$. 
The dynamics of the two ASEPs (under the basic coupling) can be described using this partial ordering $\succeq_o$. Specifically,
along every edge $e$ we sort the endpoints of $e$ according to $\succeq_o$ in increasing order at rate $1$, and in decreasing order at rate $q$. Now we consider the exceptional cases of $(2,1)$ and $(1,2)$, as well as $(2,0)$ and $(0,2)$. Whenever an edge with endpoint spins $(2,1)$ and $(1,2)$ receives an update, we obtain $(1,1)$ and $(2,2)$ as the resulting endpoints (in increasing order at rate $1$ and decreasing order at rate $q$). 
A similar statement applies to the pairs $(2,0)$ and $(0,2)$.  

Consider the following properties of the system at time $t\geq0$.
\begin{enumerate}
    \item [$\bullet$] There are at least as many pairs $(0,2)$ as $(2,0)$, and the particle pairs $(0,2)$ are all to the right of all the pairs $(2,0)$.
    \item [$\bullet$] There are at least as many pairs $(1,2)$ as $(2,1)$, and the particle pairs $(1,2)$ are all to the right of all the pairs $(2,1)$.
    \item [$\bullet$] The rightmost pair which is not of type $(1,1)$ or $(0,0)$ are within $\{ (2,2), (1,2), (0,2) \}$.
\end{enumerate}
By our assumption \eqref{eq:DominationAssumption} and $s^{\prime}>s$ for all $s \in S$ and $s^{\prime} \in S^{\prime}$, the above properties hold at time $t=0$. One can check that these  properties are preserved over time, since every time a pair $(0,2)$ (respectively $(1,2)$) gets swapped with a pair $(2,0)$ (respectively $(2,1)$), they get annihilated. The third property at time $t\geq0$ implies \eqref{eq:DominationStatement}. The argument applies mutatis mutandis for \eqref{eq:Dom2}. 
\end{proof}

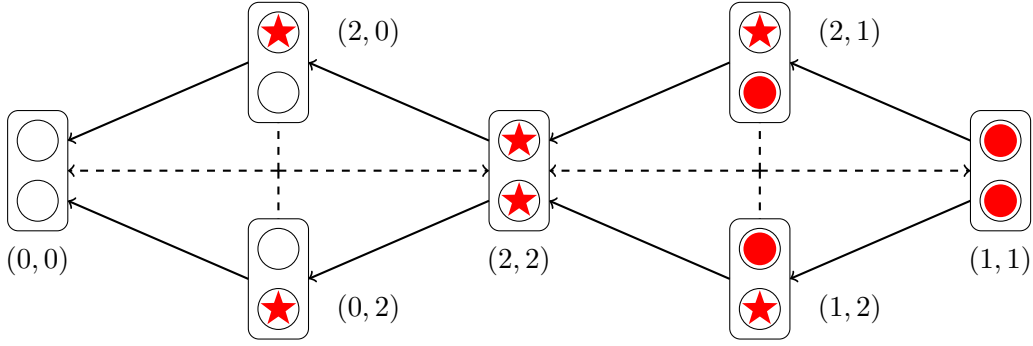
\begin{figure}
\centering
\begin{tikzpicture}[scale=0.8]

\draw[rounded corners] (-4, 0) rectangle (-3, 2);

\draw[rounded corners] (0, -1.8) rectangle (1, 0.2);

\node (X) at (2,-1.3) {$(0,2)$};

\node (X) at (2,3.3) {$(2,0)$};

\draw[rounded corners] (0, 1.8) rectangle (1, 3.8);


\node (X) at (4.5-8,-0.5) {$(0,0)$};

\draw[rounded corners] (4, 0) rectangle (5, 2);

\node (X) at (4.5,-0.5) {$(2,2)$};

\draw[rounded corners] (8, -1.8) rectangle (9, 0.2);

\node (X) at (10,-1.3) {$(1,2)$};

\draw[rounded corners] (8, 1.8) rectangle (9, 3.8);

\node (X) at (10,3.3) {$(2,1)$};

\draw[rounded corners] (12, 0) rectangle (13, 2);

\node (X) at (12.5,-0.5) {$(1,1)$};

\draw[thick,->] (8-8, -0.8) -> (5-8, 0.5);
\draw[thick,->] (8-8, 2.8) -> (5-8, 1.5);
\draw[thick,->] (12-8, 1.5) -> (9-8, 2.8);
\draw[thick,->] (12-8, 0.5) -> (9-8,-0.8);

\draw[thick,dashed] (8.5-8,1) -- (8.5-8,1.8);
\draw[thick,dashed] (8.5-8,1) -- (8.5-8,0.2);
\draw[thick, dashed, ->] (8.5-8,1) -- (5-8, 1);
\draw[thick, dashed, ->] (8.5-8,1) -- (12-8, 1);

\draw[thick,->] (8, -0.8) -> (5, 0.5);
\draw[thick,->] (8, 2.8) -> (5, 1.5);
\draw[thick,->] (12, 1.5) -> (9, 2.8);
\draw[thick,->] (12, 0.5) -> (9,-0.8);

\draw[thick,dashed] (8.5,1) -- (8.5,1.8);
\draw[thick,dashed] (8.5,1) -- (8.5,0.2);
\draw[thick, dashed, ->] (8.5,1) -- (5, 1);
\draw[thick, dashed, ->] (8.5,1) -- (12, 1);

\node[shape=circle,scale=1.5,draw] (E) at (0.5-4,0.5){} ; 

\node[shape=circle,scale=1.5,draw] (E) at (0.5-4,1.5){} ; 

\node[shape=circle,scale=1.5,draw] (E) at (0.5,2.3){} ; 
\node[shape=circle,scale=1.5,draw] (E) at (0.5,3.3){} ; 
\node[shape=star,star points=5,star point ratio=2.5,fill=red,scale=0.55] (Y1) at (0.5,3.3) {};

\node[shape=circle,scale=1.5,draw] (E) at (0.5,-1.3){} ; 
\node[shape=star,star points=5,star point ratio=2.5,fill=red,scale=0.55] (Y1) at (0.5,-1.3) {};
\node[shape=circle,scale=1.5,draw] (E) at (0.5,-0.3){} ;

\node[shape=circle,scale=1.5,draw] (E) at (4.5,1.5){} ; 
\node[shape=star,star points=5,star point ratio=2.5,fill=red,scale=0.55] (Y1) at (4.5,1.5) {};

\node[shape=circle,scale=1.5,draw] (E) at (4.5,0.5){} ; 
\node[shape=star,star points=5,star point ratio=2.5,fill=red,scale=0.55] (Y1) at (4.5,0.5) {};

\node[shape=circle,scale=1.5,draw] (E) at (8.5,-1.3){} ; 
\node[shape=star,star points=5,star point ratio=2.5,fill=red,scale=0.55] (Y1) at (8.5,-1.3) {};

\node[shape=circle,scale=1.5,draw] (E) at (8.5,-0.3){} ; 
\node[shape=circle,scale=1.2,fill=red] (Y1) at (8.5,-0.3) {};

\node[shape=circle,scale=1.5,draw] (E) at (8.5,2.3){} ; 
\node[shape=circle,scale=1.2,fill=red] (Y1) at (8.5,2.3) {};

\node[shape=circle,scale=1.5,draw] (E) at (8.5,3.3){} ; 
\node[shape=star,star points=5,star point ratio=2.5,fill=red,scale=0.55] (Y1) at (8.5,3.3) {};

\node[shape=circle,scale=1.5,draw] (E) at (12.5,1.5){} ; 
\node[shape=circle,scale=1.2,fill=red] (Y1) at (12.5,1.5) {};

\node[shape=circle,scale=1.5,draw] (E) at (12.5,0.5){} ; 
\node[shape=circle,scale=1.2,fill=red] (Y1) at (12.5,0.5) {};
	\end{tikzpicture}	
\caption{\label{fig:SecondClassHierachy}
Illustration of the particle pairs and the partial ordering between them that are involved in the proof of Lemma \ref{lem:CouplingPositions}. The partial ordering $\succeq_o$ is indicated by the directions of the solid arrows, {i.e., an arrow from a state $x$ to a state $y$ indicates that $x \succeq_o y$.} The vertical dashed line between $(2,0)$ and $(0,2)$ indicate that they are not comparable under the partial ordering. The horizontal dashed arrows indicates that {whenever an edge with endpoints $(2,0)$ and $(0,2)$ receives an update, they get replaced by $(0,0)$ and $(2,2)$, where we place $(0,0)$ at the left endpoint of the edge if the update was caused by a rate $1$ Poisson clock, and at the right endpoint of the edge if the update was caused by a rate $q$ Poisson clock.} A similar statement applies to the vertical dashed line between $(2,1)$ and $(1,2)$.
}
 \end{figure}

\subsection{Moderate deviations for second class particles}\label{sec:Moderate}

In the following, we discuss moderate deviation estimates for the positions of second class particles in the ASEP on the integers started from a Bernoulli-$\rho$-product measure. We will be interested in two asymmetric simple exclusion process on the integers according to the basic coupling, started from a pair of configurations, which differ only at a finite number of positions. At all other positions, they have the law of a Bernoulli-$\rho$-product measure. More precisely, let $A \subseteq \Z$ with $|A|< \infty$. Let $(\zeta^{1}(t))_{t \geq 0}$ and $(\zeta^{2}(t))_{t \geq 0}$ be two asymmetric simple exclusion processes on the integers 
such that
\begin{equation}\label{eq:BernoulliLaw1}
\zeta_x^{1}(0)=\zeta^{2}_x(0) \sim \textup{Bern}(\rho)
\end{equation} independently for all $x\in \Z \setminus A$, where $\textup{Bern}(\rho)$ denotes the Bernoulli-$\rho$ measure for some $\rho \in [0,1]$. For all $x\in A$, we set
\begin{equation}\label{eq:BernoulliLaw2}
\zeta_x^{1}(0)=1-\zeta^{2}_x(0) =1. 
\end{equation} 
We consider the disagreement process $(\xi(t))_{t \geq 0}$ defined by \eqref{def:Disagreement}.
Note that the number of second class particles is preserved. 
We first consider the case when $A=\{ 0 \}$, where we denote by $(Z(t))_{t \geq 0}$  
the unique position of the second class particle in $(\xi(t))_{t \geq 0}$. We have the following moderate deviation estimate for the position of the second class particle, which is a simple consequence of a recent result by Landon and Sosoe \cite{landon2023tail}. To simplify the notation, we will write  
\begin{equation}\label{def:Intervalxyz}
\mathcal{I}(t,y,z) := \big[ z + (1-q)(1-2\rho)t - y -1, z + (1-q)(1-2\rho)t + y +1  \big] . 
\end{equation} Intuitively, a second class particle in a Bernoulli-$\rho$-product measure travels at a linear speed 
\begin{equation}\label{def:kappa}
    \kappa := (1-q)(1-2\rho) , 
\end{equation}  and has fluctuations of order $t^{2/3}$ at time $t$. 
\begin{theorem}[c.f.\ Theorem~2.5 in \cite{landon2023tail}]\label{thm:Moderate}
Consider the disagreement process $(\xi(t))_{t \geq 0}$ of two ASEPs with initial conditions given by \eqref{eq:BernoulliLaw1} and \eqref{eq:BernoulliLaw2} for $A=\{0\}$, in which case 
there {is} 
a single second class particle $(Z(t))_{t \geq 0}$. Then there exists some $c_0,C_0>0$, depending only on $\rho$ and $q$, such that for all $y,T>0$ sufficiently large, with $\log(T) \leq y \leq  T^{1/3}$,
\begin{equation}\label{eq:vbrewa}
\P\left( \exists  t\in [0,T]  \, \colon \, Z(t) \notin \mathcal{I}(t,yT^{2/3},0) \right) \leq C_0 \exp( -c_0 y^{3} ) . 
\end{equation}
\end{theorem}
\begin{proof}
Theorem~2.5 in~\cite{landon2023tail} states the moderate deviation estimate as in Theorem \ref{thm:Moderate} for fixed times $s$ and $w \leq (1-q)s^{1/3}$, i.e.\ there exist constants $c_1,C_1>0$ such that
\begin{equation}\label{eq:ModerateIntegers1}
\P\left( Z(s) \notin \mathcal{I}(s,ws^{2/3},0) \right) \leq C_1 \exp( -c_1 w^3 ) . 
\end{equation} By a change of variables $\widetilde{w}=(1-q)w/4$ in \eqref{eq:ModerateIntegers1},
we see  that for all $w \in [(1-q)s^{1/3},4s^{1/3}]$, 
\begin{equation}\label{eq:ModerateIntegers2}
\P\left( Z(s) \notin \mathcal{I}(s,ws^{2/3},0) \right) \leq\P\left(Z(s) \notin \mathcal{I} (s,\widetilde{w}s^{2/3},0 ) \right)\leq C_1 \exp( - c_1^{\prime} w^3 ),
\end{equation}
where $c_1^{\prime}=c_1 ((1-q)/4)^3$. {We distinguish now whether $w \geq 4s^{1/3}$ or $w < 4s^{1/3}$. 
For $w \geq 4s^{1/3}$, note that the second class particle $(Z_t)_{t \geq 0}$ moves at rate at most $(1+q)<2$. Now stochastically dominate the jump times of $(Z_t)_{t \geq 0}$ by independent Exponential-$2$-random variables; see also Lemma~\ref{lem:CouplingOpenZ}. Then an application of the exponential Markov inequality yields that there exist constants $c_2,C_2>0$ such that for all $x\geq 1$
\begin{equation}\label{eq:PoissonBound}
    \P\left( Z(s) \notin [-4xs,4xs] \right) \leq C_2 \exp( -c_2 x  )
\end{equation}
In particular, we see that for all $s>0$ and $w \geq 4s^{1/3}$, }
\begin{equation}\label{eq:ModerateIntegers3}
    \P\left( Z(s) \notin \mathcal{I}(s,ws^{2/3},0) \right) \leq C_2 \exp( -c_2 ws^{2/3}  ) . 
\end{equation}
Combining \eqref{eq:ModerateIntegers1}, \eqref{eq:ModerateIntegers2}, \eqref{eq:ModerateIntegers3}, there exist  $c_3,C_3>0$ such that for any $ s\leq\lceil T\rceil$ and $y\leq T^{1/3}$,
\begin{equation}\label{eq:bawret}
    \P\left(  Z(s) \notin \mathcal{I}(s,yT^{2/3},0) \right) \leq C_3  \exp( -c_3 y^{3} ).
\end{equation}
We explain the above bound as follows. Setting $ws^{2/3}=yT^{2/3}$ gives $w=y(T/s)^{2/3}\geq y/2$. In the case that $w\leq 4s^{1/3}$, the bound \eqref{eq:bawret} follows easily from \eqref{eq:ModerateIntegers1} and \eqref{eq:ModerateIntegers2}. When $w \geq 4s^{1/3}$, we have 
$ws^{2/3}=yT^{2/3}\geq y^3$ since $y\leq T^{1/3}$, hence \eqref{eq:bawret} follows from \eqref{eq:ModerateIntegers3}.

Using a union bound of \eqref{eq:bawret} and then changing $y$ to $y/2$, we have
\begin{equation}\label{eq:ModerateIntegers100}
    \P\left( Z(s) \in \mathcal{I}(s,yT^{2/3}/2,0)  \mbox{ for all } s\in\llbracket\lceil T\rceil\rrbracket \right) \geq 1-C_3 \lceil T\rceil \exp( -c_3 y^{3}/8 ) . 
\end{equation}
{On the other hand, by the same arguments as in \eqref{eq:PoissonBound}, stochastically dominating the jump times of the second class particle by independent Exponential-$2$-random variables,} there exist constants $c_4,C_4>0$ such that for any interval $[a,b]$, $x\in\N$ and $s\in \N$,
\begin{equation}\label{eq:ModerateIntegers4}
\P( Z(t) \in \lBr a-x,b+x \rBr \text{ for all } t \in [s,s+1] \, | \, Z(s) \in [a,b] ) \geq 1- C_4\exp(-c_4 x ). 
\end{equation}  
We use \eqref{eq:ModerateIntegers4} for $[a,b]= \mathcal{I}(s,yT^{2/3}/2,0)$ and $x=\lfloor y^3/2\rfloor \leq yT^{2/3}/2 $. 
Combining with \eqref{eq:ModerateIntegers100},
\[
\P\left( Z(t) \in \mathcal{I}(t,yT^{2/3},0)  \mbox{ for all } t\in[0,T] \right) \geq 
1-C_3 \lceil T\rceil \exp( -c_3 y^{3}/8 )-C_4\lceil T\rceil\exp(-c_4 \lfloor y^3/2\rfloor ). 
\] In view of $y \geq \log(T)$, we conclude the proof of \eqref{eq:vbrewa}.
\end{proof}

The next result establishes a bound for the positions of a finite collection of second class particles for the open ASEP starting with a Bernoulli-$\rho$-product measure. 
\begin{lemma}\label{lem:GuideParticles}
Let $\tilde{\zeta}$ have a Bernoulli-$\rho$-product law for some $\rho \in (1/2,1)$, and recall $\kappa=\kappa(\rho,q)<0$. 
Let $S \subseteq \mathbb{Z}$ be a finite set which is allowed to depend on the random configuration $\tilde{\zeta}$. Denote $|S|=r$. Let $(\zeta(t))_{t \geq 0}$ be an ASEP on the integers with initial configuration given by $\zeta_x(0)=\tilde{\zeta}_x$ for $x\notin S$ and $\zeta_x(0)=2$ for $x \in S$. For all $i \in \lBr r \rBr$, let $(Z^{i}(t) )_{t \geq 0}$ denote the position of the $i$-$\textup{th}$ second class particle in $(\zeta(t))_{t \geq 0}$, counted from left to right, and assume that
\begin{equation}
 Z^{r}(0)  \leq 0 . 
\end{equation} 
Then there exist constants $c=c(\rho,r)>0$ and $t_0=t_0(\rho,r)>0$ such that for all $t \geq t_0$
\begin{equation}\label{eq:ModerateDeviationFinite}
\P\left( \sup_{s\in [0,t]} Z^{r}(s)  \leq |\kappa| t  \text{ and } Z^{r}(t)  \leq \frac{1}{2}\kappa t \right) \geq 1 - \exp(-c t^{1/8}) . 
\end{equation} 
\end{lemma}
\begin{proof}
Consider the initial configuration $\zeta^{\ast}$ for the process $(\zeta(t))_{t \geq 0}$, where we place the second class particles at locations $L:=\{ j t^{3/4} \, \colon \, j \in \lBr r t^{1/8}\rBr \}$. On all other sites, we let $\zeta^{\ast}$ agree with the configuration~$\tilde{\zeta}$. Since $\rho \in (1/2,1)$, note that with probability at least $1-\exp(-c_0 t^{1/8})$ for some $c_0=c_0(\rho,r)>0$, we get that assumption \eqref{eq:DominationAssumption} from Lemma \ref{lem:CouplingPositions}
holds. Thus, it suffices to show \eqref{eq:ModerateDeviationFinite} with respect to the initial configuration $\zeta^{\ast}$ with $|L| $ many the second class particles initially at locations $L$. 
To do so, consider $|L|$ many asymmetric simple exclusion processes $(\zeta^{(j)}(t))_{t \geq 0}$, where we only keep the second class particle at $jt^{3/4}$. We denote its position by $(\tilde{Z}^{j}(t) )_{t \geq 0}$ for $j\in \lBr |L| \rBr$ and choose the remaining configuration according to $\tilde{\zeta}$. We define the event 
\begin{equation}\label{eq:ModerateIndividual}
\mathcal{A}=\bigcap_{j=1}^{|L|} \mathcal{A}_j\quad\mbox{where}\quad\mathcal{A}_j := \left\{  \tilde{Z}^{j}(s)  \in \mathcal{I}(s,t^{2/3+1/24},jt^{3/4}) \text{ for all } s \in [0,t] \right\} .
\end{equation} 
Note that for some constant $c_1>0$ and for all $t>0$ sufficiently large,
\begin{equation}\label{eq:TheeventA}
\P (\mathcal{A}) \geq 1-\exp(-c_1 t^{1/8})
\end{equation} holds by a union bound and Theorem \ref{thm:Moderate}. Observe that under the basic coupling $\mathbf{P}$, we get that
\begin{equation}
\mathbf{P}(  Z^{j}(s) =\tilde{Z}^{j}(s)   \text{ for all } j \in \lBr rt^{1/8} \rBr \text{ and } s \in [0,t] \, | \, \mathcal{A} ) = 1, 
\end{equation} 
since, on the event $\mathcal{A}$, the second class particles maintain positive distance away from each other over time interval $[0,t]$.
Using the definition of the events $\mathcal{A}_j$ in \eqref{eq:ModerateIndividual}, we see that on the event $\mathcal{A}$,
\begin{equation}
    \sup_{s\in [0,t]} Z^{r}(s)  \leq C_1 t^{7/8}  \text{ and } Z^{r}(t) \leq \kappa t + C_1 t^{7/8}
\end{equation} for sufficiently large constant $C_1$. 
Choosing $t_0$ sufficiently large, this yields the desired result.
\end{proof}

We will make use of the following so-called microscopic concavity coupling of two ASEPs with second class particles, introduced by Balazs and Seppäläinen in \cite{balazs2009fluctuation}. This coupling preserves the order of second-class particles for component-wise ordered initial configurations.

\begin{theorem}[c.f.\ Theorem 3.1 in \cite{balazs2009fluctuation}]\label{thm:BSCoupling}
  Fix some finite set $S \subseteq \Z$. Let $(\zeta^1(t))_{t \geq 0}$ and $(\zeta^{2}(t))_{t \geq 0}$ be two ASEPs on the integers with initial configurations $\zeta^1$ and $\zeta^{2}$ such that
  \begin{equation}
      \zeta^1_x \geq \zeta^{2}_x 
  \end{equation} with $\zeta^1_x, \zeta^{2}_x \in \{0,1\}$ for all $x\notin S$, and $ \zeta^1_x = \zeta^{2}_x=2$ for all $x \in S$.
Enumerate the second class particles from left to right in $(\zeta^1(t))_{t \geq 0}$ by $(Z^{i,1}(t) )_{t \geq 0}$ for $i\in \lBr |S| \rBr$, and by $(Z^{i,2}(t) )_{t \geq 0}$ for $i\in{\lBr |S| \rBr}$ in $(\zeta^{2}(t))_{t \geq 0}$. Then there exists a coupling $\bar{\mathbf{P}}$ such that
\begin{equation}
     \bar{\mathbf{P}} \left(  Z^{i,2}(t) \geq Z^{i,1}(t)  \text{ for all } i \in \lBr |S| \rBr \text{ and }  t\geq 0 \,  \big|  \, \zeta^1(0)=\zeta^1 \text{ and } \zeta^{2}(0) = \zeta^{2} \right) =1 .
\end{equation}
\end{theorem}
\begin{remark}
    As pointed out in \cite{balazs2009fluctuation}, the coupling $\bar{\mathbf{P}}$ has to be different from the basic coupling, since the latter only preserves the order of second class particles when $q=0$.
\end{remark}

The next result 
compares the positions of the light particles in the open ASEP with different initial configurations, under the assumptions  
\begin{equation}\label{eq:RestatedAssumptions}
    A > 1  \quad \text{and} \quad AC \leq 1 .
\end{equation}
This will be our key tool to control the motion of light particles in the upcoming subsections. { Recall that $\rho=A/(1+A)$ in the high density phase. Moreover, recall that $\mu_N$ denotes the stationary distribution of the open ASEP without light particles, and that $\kappa=(1-q)(1-2\rho)$ from \eqref{def:kappa}. }
\begin{lemma}\label{lem:PushForward}
Assume that \eqref{eq:RestatedAssumptions} holds, and let $\eta \sim \mu_N$. Consider an open ASEP $(\tau(t))_{t \geq 0}$ with $r$ many light particles at locations $(\lc_i(t))_{t \geq 0}$ for all $i \in \lBr r \rBr$. Then there exists some $t_0=t_0(A,C,r)$ and $c=c(A,C,r)$ such that for all $t \geq t_0$, and all $0 < s_1 < s_2 < \dots <s_r$ with some $s_r \in \lBr N-4t \rBr$, 
\begin{equation*}
    \P\left( \lc_r(t) \leq \max\left(s_r - \frac{|\kappa|}{2} t,4t\right) \, \Big| \,  \lc_i(0) = s_i \, \mbox{ }\forall i \in \lBr r \rBr ,\mbox{ } \,  \tau_x=\eta_x \, \mbox{ } \forall x \notin \{s_i\}_{{i\in \lBr r\rBr}} \right) \geq 1  - \exp(-c t^{1/8})  ,
\end{equation*} provided that $N$ is sufficiently large.
\end{lemma}
\begin{proof}
First, we claim that without loss of generality, we can assume that $s_1 \geq 2t$. To see this, note that the event $\lc_r(t) \leq \max\left(s_r -  |\kappa| t/2,4t\right)$ does with probability at least $1-\exp(-c_1 t)$ for some constant $c_1>0$ not depend on the choice of $(\tau_x(0))_{x \leq 2t}$. More precisely, let  $(\tau^{\prime}(t))_{t \geq 0}$ be an open ASEP with some initial locations $(s_i^{\prime})_{i \in \lBr r \rBr}$ for the light particles such that $s_1^{\prime}\geq 2t$, and  $s_i^{\prime}=s_i$ whenever $s_i>2t+r$. On the remaining sites $x$, we let $\tau_x(0)=\eta_x$.
{Then the corresponding open ASEPs under the basic coupling $\mathbf{P}$ satisfy} 
\begin{equation}
    \mathbf{P}\left(  \tau_x(s)= \tau^{\prime}_x(s) \text{ for all } x \geq 4t  \text{ and } s\in[0,t]\right) \geq 1- \exp(-c_1 t)
\end{equation} for some $c_1>0$ and $t\geq t_0$ sufficiently large. {This is because on the complement event, a light particle must jump from $\llbracket0,2t+r\rrbracket$ to $\llbracket 4t,N\rrbracket$ within time interval $[0,t]$, and hence its probability can be bounded using a similar argument as in Lemma \ref{lem:CouplingOpenZ}, where we stochastically dominate the jump times of the second class particles by i.i.d. exponential random variables.} Since the event $\big\{ \lc_r(t) \leq \max\left(s_r - |\kappa| t/2,4t\right) \big\}$ is measurable with respect to $(\tau_x(t))_{x \geq 4t}$, this yields the claim.

Now we let $\tilde{\zeta} \in \{0,1\}^{\Z}$ be chosen according to a Bernoulli-$\rho$-product law, where we set $\rho=A/(A+1)$. Using Lemma \ref{lem:CouplingOpenZ}, we can couple the open ASEP with light particles to an ASEP $(\zeta(t))_{t \geq 0}$ on the integers with initial configuration given by
\begin{equation}
    \zeta_x(0)= \begin{cases} 2 & \text{ if } x=s_i \text{ for some } i \in \lBr r \rBr \\
        \eta_x  & \text{ if }  x \in \lBr N \rBr \setminus  \{ s_i \, \colon \, i \in \lBr r \rBr \} \\
        \tilde{\zeta}_x & \text{ if } x \leq 0 \text{ or } x > N 
    \end{cases}
\end{equation} for all $x \in \Z$, such that, with probability at least $1-\exp(-c_2 t)$,
\begin{equation}
  \tau_y(s)=\zeta_y(s)  
\end{equation} for all $y \in [2t,N-2t]$ and $s\in [0,t]$.
Since $AC \leq 1$, using Lemma \ref{lem:StochasticDomination}, we see that $\mu_N$ is stochastically dominating a Bernoulli-$\rho$-product measure, i.e.\ there exists a coupling such that
\begin{equation}
    \eta_x \geq \tilde{\zeta}_x
\end{equation} for all $x \in \lBr N \rBr$. By the microscopic concavity coupling stated as Theorem \ref{thm:BSCoupling}, we only need to bound the location of the $r$-th second class particle for the ASEP on the integers at time $t$, started from the initial condition with second class particles $0 < s_1 < s_2 < \dots <s_r$ within the Bernoulli-$\rho$-product measure. This bound follows from Lemma \ref{lem:GuideParticles}  shifted rightwards by $s_r$.
\end{proof}

\subsection{Hitting times for light particles}\label{sec:EscapeLight}

In this subsection, we argue that for all sufficiently large times, the light particles for the open ASEP in the high density phase will be located at a finite window around the left boundary. Using Lemma~\ref{lem:PushForward}, we first describe a way to ensure that the light particles $(\lc_i(t))_{t \geq 0}$ for $i\in \lBr r \rBr$ are sufficiently far away from the right boundary with positive probability. Throughout this subsection, we assume that $A>1$ and $AC \leq 1$ holds. 

\begin{proposition}\label{pro:EscapeLight}
Assume \eqref{eq:RestatedAssumptions}.
Let $\eta \sim \mu_N$ and let $0<s_1<s_2<\dots <s_r \leq N$ for some finite $r\in \N$. Consider the open ASEP $(\tau(t))_{t \geq 0}$ with $r$ many light particles and initial configuration {
\begin{equation}
\tau_x(0) = \begin{cases}
\ha  & \text{ if } x = s_i \text{ for some } i \in \lBr r \rBr  \\
\eta_x & \text{ if } x \notin S= \{ s_i \colon i \in \lBr r\rBr \} 
\end{cases} 
\end{equation} }for all $x\in \lBr N \rBr$.
Then there exists a constant $C_0 =C_0 (A,C,r)>0$ such that for every $\theta \in (0,1)$
\begin{equation}\label{eq:FinalGoalEscape}
\P\left(  \lc_r( C_0 N^{\theta}\log(N)) \leq N - N^{\theta}  \right) \geq 1 - \frac{1}{N}
\end{equation} for all $N$ sufficiently large. 
\end{proposition}

\begin{proof} We first claim that it suffices to show that there exists some constants $c_1,C_1>0$, depending only on $A,C,r$, such that for an arbitrary choice of $(s_i)_{i \in \lBr r \rBr}$
\begin{equation}\label{eq:GoalEscape}
    \P\left(  \lc_r( C_1 N^{\theta}) \leq N - N^{\theta}  \right) \geq c_1  . 
\end{equation}
To see that \eqref{eq:GoalEscape} indeed implies \eqref{eq:FinalGoalEscape}, we make the following observation. Since the law $\mu_N$ is invariant for the open ASEP without light particles, by Remark \ref{rem:ProjectionOpenASEP},  we can iterate \eqref{eq:GoalEscape} $C_2 \lceil\log(N)\rceil$ many times for a sufficiently large constant $C_2=C_2(c_1) \in \N$ to get 
\begin{equation}\label{eq:GoalEscapeNew}
    \P\left(  \lc_r(  i C_1 N^{\theta}) \leq N - N^{\theta}  \text{ for some } i \in \lBr C_2 \lceil \log(N) \rceil\rBr \right) \geq  1- \frac{1}{2N} . 
\end{equation} Suppose that the event in \eqref{eq:GoalEscapeNew} holds for some random $i\in \lBr C_2 \lceil \log(N) \rceil\rBr$. Then to conclude \eqref{eq:FinalGoalEscape} for $C_0=C_2+1$, we apply Lemma \ref{lem:PushForward} and  Remark \ref{rem:ProjectionOpenASEP} over $4(C_2+1)\lceil\log(N)\rceil$ many time intervals and then use a union bound. These time intervals are obtained by dividing $[i C_1 N^{\theta},(C_2+1)N^{\theta}\log(N)]$ arbitrarily into subintervals of lengths between $N^{\theta}/5$ and $N^{\theta}/4$.



In order to show \eqref{eq:GoalEscape}, we claim that without loss of generality, we can assume  that
\begin{equation}\label{eq:AbitAway}
    s_r \leq N - C_3
\end{equation}
for some constant $C_3>0$, which can be chosen arbitrarily large (but fixed). To see this, let  $(\tau^{\prime}(t))_{t \geq 0}$ be an open ASEP with some initial locations $(s_i^{\prime})_{i \in \lBr r \rBr}$ for the light particles such that $s_i^{\prime}=s_i$ whenever $s_i<N-C_3-r$, and $s_i^{\prime} \in \lBr N-C_3-r,N-C_3\rBr$ for all other $i \in \lBr r \rBr$. On the remaining sites $x$, we let $\tau_x(0)=\eta_x$.
Now the basic coupling $\mathbf{P}$ ensures that 
\begin{equation}
    \liminf_{N \rightarrow \infty} \mathbf{P}( \tau(1)=\tau^{\prime}(1) ) > 0
\end{equation}
as this requires to modify $\tau(0)$ and $\tau^{\prime}(0)$ only at positions $ \lBr N-C_3-r,N \rBr$. Since we want to show that \eqref{eq:GoalEscape} holds with positive probability,  this yields the desired claim.

In the following, we will achieve \eqref{eq:GoalEscape} by a multi-scale argument. Set $K:=\lceil 8 \kappa^{-1} \rceil$ and $M:=\theta \log_2 N$. To simplify notation, we will assume without loss generality that $M\in \N$. For all $i\in\llbracket M\rrbracket$ and $j \in \lBr K \rBr$, we define a family of time scales $(\mathcal{T}_{i}^{j})$ and spatial scales $(\mathcal{S}_{i}^{j})$ by
\begin{align}
\begin{split}\label{def:Scales}
\mathcal{T}_{i}^{j} &:= 2^{i-2} (K + j )  \\
\mathcal{S}_{i}^{j} &:= 2^{i} (1+ K^{-1} j) . 
\end{split}
\end{align}   
Moreover, we define the events
\begin{equation}
\mathcal{B}^{j}_i = \left\{   \lc_r(\mathcal{T}_{i}^{j}) \leq N -  \mathcal{S}_{i}^{j}  \right\} 
\end{equation} that the rightmost light particle is to the left of location $N-\mathcal{S}_{i}^{j}$ at time $\mathcal{T}_{i}^{j}$. 
Note that 
\[
\mathcal{T}_{i}^{j}-\mathcal{T}_{i}^{j-1}=2^{i-2}\geq\frac{2}{\kappa}\left(\mathcal{S}_{i}^{j}-\mathcal{S}_{i}^{j-1}\right) 
,\quad 
\mathcal{T}_{i}^{0}-\mathcal{T}_{i-1}^{K-1}=2^{i-3}\geq\frac{2}{\kappa}\left(\mathcal{S}_{i}^{0}-\mathcal{S}_{i-1}^{K-1}\right),
\]
and $\min\left(\mathcal{S}_i^j,N-\mathcal{S}_i^j\right)\geq 4\cdot 2^{i-2}=4\max(\mathcal{T}_{i}^{j}-\mathcal{T}_{i}^{j-1},\mathcal{T}_{i}^{0}-\mathcal{T}_{i-1}^{K-1})$ when $N$ is sufficiently large.

We now claim that there exist constants $c_2,C_4>0$ such that for all $i \in \lBr C_4, M \rBr$ and $j \in \lBr K-1 \rBr$,
\begin{equation*}
\P(\mathcal{B}^{j}_i \, |  \, \mathcal{B}^{j-1}_i ) \geq   \inf_{(s_k) \colon s_r \leq N-\mathcal{S}^{j-1}_{i}}\P\left( \lc_r(2^{i-2}) \leq N - \mathcal{S}^j_{i} \, \big| \,  \lc_k(0)=s_k \, \forall  k\in \lBr r \rBr\right)
\geq 1 - \exp(-c_2 2^{i/8}) ,
\end{equation*} 
and similarly, for all $i \in \lBr C_4, M \rBr$,
\begin{equation*}
    \P(\mathcal{B}^{0}_i \, |  \, \mathcal{B}^{K-1}_{i-1} ) \geq  \inf_{(s_k) \colon s_r \leq N-\mathcal{S}^{K-1}_{i-1}}\P\left( \lc_r(2^{i-3}) \leq N - \mathcal{S}^0_{i} \, \big| \,  \lc_k(0)=s_k \, \forall  k\in \lBr r \rBr\right)
\geq 1 - \exp(-c_2 2^{i/8}). 
\end{equation*}  For the first inequalities, respectively, we use the strong Markov property together with Lemma \ref{lem:LightProjection} as $\eta$ is chosen according to the stationary measure $\mu_N$ of the open ASEP without light particles; see Remark \ref{rem:ProjectionOpenASEP}. For the second inequalities, we use Lemma \ref{lem:PushForward}. 
Thus, we obtain that
\begin{multline}\label{eq:vbersa}
\P\left( \lc_r(KN^{\theta}/4)\leq N-N^{\theta} \Big{|} \lc_r(K2^{C_4}/4)\leq N-2^{C_4}\right)=\P\left( \mathcal{B}_{M}^{0} | \mathcal{B}_{C_4}^{0}  \right) \\\geq  1 - K \sum_{i=C_4}^{M} \exp(-c_2 2^{i/8}) > \frac{1}{2},
\end{multline} where the last inequality above follows from choosing the constant $C_4>0$ sufficiently large.
Recall our assumption \eqref{eq:AbitAway} that $s_r\leq N-C_3$. Choosing $C_3 = (K+1) 2^{C_4}$ and using again Lemma \ref{lem:PushForward}, we have
$\lc_r(K2^{C_4}/4)\leq \max(s_r,K2^{C_4})\leq N-2^{C_4}$ with positive probability, for $N$ sufficiently large. 
Then in view of \eqref{eq:vbersa} we conclude the desired result \eqref{eq:GoalEscape}, where $C_1=K/4$. 
\end{proof}

Next, we show that with probability tending to $1$ as $N$ goes to infinity, within time of order $N$ we will reach a window of size $N^{\theta}$ near the left boundary, for some $\theta \in (0,1)$.

\begin{lemma}\label{lem:Traverse}
Assume \eqref{eq:RestatedAssumptions}.
Let $\eta \sim \mu_N$ and fix some $0<s_1<s_2<\dots <s_r \leq N-N^{\theta}$ with some finite $r\in \N$ and parameter $\theta \in (0,1)$. Consider the open ASEP $(\tau(t))_{t \geq 0}$ with $r$ many light particles and with initial condition
\begin{equation}
\tau_x(0) = \begin{cases}
\ha  & \text{ if } x = s_i \text{ for some } i \in \lBr r \rBr  \\
\eta_x & \text{ if } x \notin S= \{ s_i \colon i \in \lBr r\rBr \} 
\end{cases}
\end{equation} for all $x\in \lBr N \rBr$.
Then there exists a constant $C_1 =C_1 (A,C,r)>0$ such that
\begin{equation}\label{eq:HitRatherClose}
\P\left(  \lc_r( C_1 N ) \leq N^{\theta} \right) \geq 1 - \frac{1}{N}
\end{equation} for all $N$ sufficiently large. 
\end{lemma}
\begin{proof}
For the first statement, recall $K = \lceil 8 \kappa^{-1} \rceil$.
To simplify the notation, we assume without loss of generality that $KN^{1-\theta}\in\N$.
For all $i\in \llbracket K, K N^{1-\theta}-K\rrbracket$, we define the events 
\begin{equation}
    \mathcal{C}_i := \left\{ \lc_r\left( \frac{i}{4} N^{\theta} \right) \leq N -  \frac{i}{K}   N^{\theta}  \right\}.
\end{equation}
We get from Lemma \ref{lem:PushForward} that there exists some constant $c_2>0$ such that for all $i\geq 1$,
\begin{equation}
  \P( \mathcal{C}_i \, | \, \mathcal{C}_{i-1}) \geq 1 - \exp(-c_2N^{\theta/8}),
\end{equation} 
where we mention that we are again using Remark \ref{rem:ProjectionOpenASEP}.
Thus we have
\begin{multline}\label{eq:pin 1}
\P( \lc_r((KN-KN^{\theta})/4)\leq N^{\theta} |  \lc_r(K N^{\theta}/4)\leq N-N^{\theta} ) = \P(\mathcal{C}_{K N^{1-\theta}-K} | \mathcal{C}_K )\\ \geq 1 - KN^{1-\theta}\exp(-c_2N^{\theta/8}) .
\end{multline}
For $j\in\llbracket K\rrbracket$, using Lemma \ref{lem:PushForward} and Remark \ref{rem:ProjectionOpenASEP}, we have
\begin{equation}\label{eq:pin 2}
\P\left(\lc_r(jN^{\theta}/4)\leq N-N^{\theta}\big{|} \lc_r( (j-1)N^{\theta}/4)\leq N-N^{\theta}\right)\geq 1-\exp(-c_2 N^{\theta/8}),
\end{equation}
and
\begin{equation}\label{eq:pin 3}
\P\left(\lc_r((KN-(j-1)N^{\theta})/4)\leq N^{\theta} \big{|} \lc_r((KN-jN^{\theta})/4)\leq N^{\theta} \right) \geq 1-\exp(-c_2 N^{\theta/8}).
\end{equation}
Recall our assumption $\lc_r(0)=s_r\leq N-N^{\theta}$. We use union bounds for $j\in\llbracket K\rrbracket$ respectively in \eqref{eq:pin 2} and \eqref{eq:pin 3} and then use \eqref{eq:pin 1}, in summary we have
\[
\P\left(\lc_r(KN/4)\leq N^{\theta}\right)\geq 1 - KN^{1-\theta}\exp(-c_2N^{\theta/8}) -2K \exp(-c_2 N^{\theta/8}).
\]
We conclude \eqref{eq:HitRatherClose} by taking $C_1=K/4$.
\end{proof}



Next, we use a similar multi-scale strategy as in the proof of Proposition \ref{pro:EscapeLight} to show that for all sufficiently large times $t$, the light particles are with high probability located in a finite window around the left boundary of the segment. 

\begin{lemma}\label{lem:HittingWindow}
Assume \eqref{eq:RestatedAssumptions}.
Let $\eta \sim \mu_N$ and fix some $0<s_1<s_2<\dots <s_r \leq N$ with some finite $r\in \N$. Consider the open ASEP $(\tau(t))_{t \geq 0}$ with $r$  light particles and initial configuration
\begin{equation}
\tau_x(0) = \begin{cases}
\ha  & \text{ if } x = s_i \text{ for some } i \in \lBr r \rBr  \\
\eta_x & \text{ if } x \notin S= \{ s_i \colon i \in \lBr r\rBr \} 
\end{cases}
\end{equation} for all $x\in \lBr N \rBr$.
Then there exists some constant $C_0=C_0(A,C,r)>0$ such that for all $\varepsilon>0$, there exists some $n=n(\varepsilon)\in \N$ such that for all $t \geq C_0 N$,
\begin{equation}\label{eq:light particles in finite window}
\P( \lc_1(t) , \lc_2(t), \dots, \lc_r(t) \in \lBr n \rBr )  \geq 1 - \varepsilon ,
\end{equation} provided that $N$ is chosen sufficiently large.
\end{lemma}
\begin{proof} Fix some $\theta \in (0,1)$, and recall 
$K:=\lceil 8 \kappa^{-1} \rceil$ and $M=M(\theta):=\theta \log_2 N$. We assume without loss of generality that $M\in \N$ and recall the scales  $(\mathcal{T}_{i}^{j})$ and $(\mathcal{S}_{i}^{j})$ from \eqref{def:Scales}.
Let $C_1\in\N$ be a constant, which we will determine later on. We define the time scale $(\widehat{T}_{i}^{j})$ by the relation
\[
\widehat{\mathcal{T}}_{i}^{j} := \mathcal{T}_{M}^{0} - \mathcal{T}_{i}^{j}
\]
for all $i \in \lBr C_1,M \rBr$ and all $j \in \lBr K \rBr$, as well as the events
\[
\mathcal{D}^{j}_i := \left\{  \lc_r(\widehat{\mathcal{T}}_{i}^{j}) \leq  \mathcal{S}_{i}^{j}  \right\} .
\]
Note that  the same arguments as in the proof of Proposition \ref{pro:EscapeLight}  yield that there exist constants $c_1,C_2,C_3>0$ such that
\begin{align*}
    \P\left( \mathcal{D}^{j}_i \, \big| \, \mathcal{D}^{j+1}_i \right)&= \P\left( \lc_r(2^{i-2}) \leq  \mathcal{S}_{i}^j  \, \big| \, \lc_r(0)\leq \mathcal{S}^{j+1}_{i} \right)  \geq 1 - C_2 \exp(-c_1 2^{i/8} ) \\
    \P\left( \mathcal{D}^{K-1}_{i-1} \, \big| \, \mathcal{D}^{0}_i \right)&= \P\left( \lc_r(2^{i-3}) \leq  \mathcal{S}_{i-1}^{K-1}  \, \big| \, \lc_r(0)\leq \mathcal{S}_{i}^0 \right)  \geq 1 - C_2 \exp(-c_1 2^{i/8} )
\end{align*} for all $i \in \lBr C_3,M \rBr $ and all $j \in \lBr K \rBr$. Choosing now $C_1=C_1( c_1,C_2,C_3)$ sufficiently large, we get
\begin{equation}\label{eq:HitSmall}
   \P\left(\lc_r( \mathcal{T}_{M}^{0} - \mathcal{T}_{C_1}^{0}) \leq  2^{C_1} \, \big| \, \lc_r( 0) \leq  N^{\theta} \right)
   = \P\left(\mathcal{D}_{C_1}^0(t)\big{|}\mathcal{D}_M^0(t)\right)
   \geq 1 - K \sum_{i=C_1}^{M}  C_2 \exp(-c_1 2^{i/8} ) \geq 1-  \frac{\varepsilon}{2},
\end{equation}
for sufficiently large $N$. Suppose that we start the process at time $s\geq0$ with arbitrary locations of light particles $0<s_1<s_2<\dots <s_r \leq N$ within $\mu_N$. Using Proposition \ref{pro:EscapeLight}, after time $C_4 N^{\theta}\log(N)$, with probability at least $1-1/N$, the rightmost light particle $\lc_r$ gets in the window $[0, N-N^{\theta}]$. Then using Lemma \ref{lem:Traverse}, after another time $C_5N$, with probability at least $1-1/N$, $\lc_r$ gets in the window $[0, N^{\theta}]$. Finally we use \eqref{eq:HitSmall} above, after again time $\mathcal{T}_{M}^{0} - \mathcal{T}_{C_1}^{0}=N^{\theta}-2^{C_1}$, with probability at least $1-\varepsilon/2$, $\lc_r$ gets in the window $[0, n]$. 
In summary, at time 
\[
s+C_4 N^{\theta}\log(N)+C_5N+N^{\theta}-2^{C_1},
\]
the rightmost light particle is located in $[0,n]$ with probability at least $(1-1/N)^2(1-\varepsilon/2)\geq1-\varepsilon$ for $N$ sufficiently large. By choosing $s\geq0$ arbitrarily, we conclude the proof.
\end{proof}

\begin{proof}[Proof of Theorem \ref{thm:mainHighLow}]
As mentioned in the beginning of this section, we only consider the fan region of the high density phase $A>1$, $AC\leq1$, as the result for the low density case follows from the well-known particle-hole symmetry.

We fix an arbitrary sequence $0<s_1<s_2<\dots <s_r \leq N$ and let $\eta \sim \mu_N$. Consider the open ASEP with light particles process $(\tau(t))_{t\geq0}$ on $\llbracket N\rrbracket$ starting with the initial configuration 
\begin{equation*}
\tau_x(0) = \begin{cases}
\ha  & \text{ if } x = s_i \text{ for some } i \in \lBr r \rBr  \\
\eta_x & \text{ otherwise } 
\end{cases}
\end{equation*} for all $x\in \lBr N \rBr$. By the Markov chain convergence theorem, as $t\rightarrow\infty$, the random variables $\lc_i(t)$ converge to $\lc_i$ under the stationary measure $\mt_{N,r}$, for all $i\in\llbracket r\rrbracket$. By taking $t\rightarrow\infty$ in \eqref{eq:light particles in finite window} from Lemma \ref{lem:HittingWindow}, we know that for any $\varepsilon>0$ there exists $n=n(\varepsilon)$ such that 
\[
\mt_{N,r}( \lc_1 , \lc_2 , \dots, \lc_r \in \lBr n \rBr )  \geq 1 - \varepsilon 
\]
for sufficiently large $N$. Taking $N\rightarrow\infty$ and then $\varepsilon\rightarrow0$, we conclude the proof.
\end{proof}

\subsection{Mixing times for the open ASEP with light particles in the high density phase}\label{sec:MixingTimesLight}

Using the previous estimates on the light particles, we can now prove Theorem \ref{thm:MainMixingTime} on the mixing time of the open ASEP with light particles in the fan region of the high density phase. 
We have the following strategy. First, we ensure that the color projection of the open ASEP, defined in Section~\ref{sec:SecondClassProjection}, has mixed, using the results from \cite{gantert2023mixing}. We then use the estimates from Section~\ref{sec:EscapeLight} to couple the light particles. 
We start by recalling a result by Gantert et al.\ from \cite{gantert2023mixing} on the open ASEP without light particles in the high density phase. 

\begin{theorem}[c.f.\ Theorem 1.5 in \cite{gantert2023mixing}]\label{thm:MixingTimesOpenASEP} Let  $(\eta(t))_{ t\geq 0}$ and  $(\eta^{\prime}(t))_{ t\geq 0}$ be two open ASEPs without light particles under to the basic coupling $\mathbf{P}$. Assume that $A > \max(1,C)$ and $AC \leq 1$. Then there exists some $C_0=C_0(A,C)>0$ such that for any pair of initial configurations $\tilde{\eta},\tilde{\eta}^{\prime}$,
\begin{equation}
\lim_{N \rightarrow \infty} \mathbf{P}\left( \eta(t) =  \eta^{\prime}(t) \text{ for all } t \geq C_0 N \, \big| \, \eta(0)=\tilde{\eta} \text{ and } \eta^{\prime}(0)=\tilde{\eta}^{\prime} \right) = 1 . 
\end{equation}
\end{theorem}

Let us point out that strictly speaking, \cite{gantert2023mixing} states in Theorem 1.5 only bounds on the mixing time of the open ASEP. However, a closer inspection of the arguments in Section 7 of \cite{gantert2023mixing} reveals that the authors give a bound on the coalescence time when starting from an arbitrary initial configuration (using Lemma~\ref{lem:sandwiching} to reduce the statement to the extremal configurations), and then use Corollary 2.4 in \cite{gantert2023mixing} to convert Theorem \ref{thm:MixingTimesOpenASEP} into an upper bound on the mixing time. Next, we establish a coupling result for pairs of stationary configurations, which only differ in the location of the light particles.


\begin{lemma}\label{lem:Premixing}
Fix $r\in \N$ and two sets $S,S^{\prime} \subseteq \lBr N \rBr$ such that $|S|=|S^{\prime}|=r$. Let $(\tau(t))_{t \geq 0}$ and $(\tau^{\prime}(t))_{t \geq 0}$ be two open ASEPs with $r$ many light particles on a segment of length $N$ such that
\begin{align*}
   \tau_x =  \begin{cases}
       \ha & \text{ if } x\in S \\
       \eta_x & \text{ if } x \notin S
   \end{cases} \qquad \text{and} \qquad 
   \tau^{\prime}_x =  \begin{cases}
       \ha & \text{ if } x\in S^{\prime} \\
       \eta_x & \text{ if } x \notin S^{\prime} 
   \end{cases}
\end{align*}
for all $x\in \lBr N \rBr$, where $\eta \sim \mu_N$. Then there exists a constant $C_1>0$ such that  under the basic coupling $\mathbf{P}$, and an arbitrary choice of $S$ and $S^{\prime}$,
\begin{equation}
\lim_{N \rightarrow \infty}\mathbf{P}\left( \tau(C_1 N)=\tau^{\prime}(C_1 N) \right) = 1 .
\end{equation} 
\end{lemma}
\begin{proof}
Let $(\lc_i(t))_{t\geq 0}$ and $(\lc^{\prime}_i(t))_{t\geq 0}$ for all $i\in \lBr r \rBr$ denote the locations of the light particles in $(\tau(t))_{t \geq 0}$ and $(\tau^{\prime}(t))_{t \geq 0}$, respectively.  Note that under the basic coupling, the two processes $(\tau(t))_{t \geq 0}$ and $(\tau^{\prime}(t))_{t \geq 0}$ can at any time $t$ only differ at the locations of the light particles. 

Let $C_2,C_3>0$ be two constants, which we will determine later on. 
In order to couple the locations of the light particles, we define for all $k \in \lBr N \rBr$ the events
\begin{equation}
   \mathcal{B}^{(k)} := \left\{ \lc_r(t) , \lc_r^{\prime}(t)  \leq \sqrt{N} \text{ for all } t=C_2 N + C_3 j \sqrt{N}   \text{ with } j \in \lBr k \rBr \right\} .
\end{equation} 
Using Proposition~\ref{pro:EscapeLight} and  Lemma \ref{lem:Traverse} with $\theta=1/2$, we see that $C_2>0$ can be chosen such that 
\begin{equation}\label{eq:Bvents1}
\mathbf{P}(\mathcal{B}^{(1)}) \geq 1 - \frac{1}{N} 
\end{equation} for all $N$ sufficiently large. Moreover, note that there exists some constant $c_1>0$ such that 
  \begin{equation}\label{eq:LongerStay}
 \P\left(  \lc_r( C_2 N + C_3 j N^{\theta} ) \leq N^{\theta} \, \Big| \, \lc_r( C_1 N + C_3 (j-1) N^{\theta} ) \leq N^{\theta} \right) \geq 1 - \lceil 4 C_3 \rceil\exp(- c_1 N^{\theta/8})
   \end{equation} for all $j\in \N$, and all $N$ sufficiently large. 
This follows by applying Lemma \ref{lem:PushForward} between times $C_2 N + (j+(j^{\prime}-1)/4) N^{\theta} $ and $C_1 N + C_2 (j+j^{\prime}/4) N^{\theta}$ for $j^{\prime} \in \lBr \lceil 4 C_3 \rceil \rBr$, using again Remark \ref{rem:ProjectionOpenASEP}.
In particular, we get that 
\begin{equation}\label{eq:Bvents2}
    \mathbf{P}(\mathcal{B}^{(k)} \, |  \,  \mathcal{B}^{(k-1)} ) \geq 1 - 2 \lceil 4 C_3 \rceil \exp( -c_1 N^{1/16}) 
\end{equation} provided that $N$ is sufficiently large. 
Next, we define the events 
\begin{align*}
\mathcal{C}^{(k)} &:= \left\{  \tau(C_2 N + C_3 (k+1) \sqrt{N})=\tau^{\prime}(C_2 N + C_3 (k+1) \sqrt{N}) \right\}  . 
\end{align*}  
We claim that $C_3$ can be chosen such that for some constant 
$c_2>0$, and all $k\in \lBr N \rBr$, we get that
\begin{equation}\label{eq:Cvent}
    \mathbf{P}\left( \mathcal{C}^{(k)} \, \big| \,  \mathcal{B}^{(k)} , (\mathcal{C}^{(j)})^{\complement} \text{ for all } j\leq k-1 \right) \geq c_2 . 
\end{equation} 
To show \eqref{eq:Cvent}, note that by Lemma \ref{lem:LightProjection}, the law of the color projection of the processes $(\tau(t))_{t \geq 0}$ and $(\tau^{\prime}(t))_{t \geq 0}$ does not depend on the location of the light particles, and thus is given by $\mu_N$ at any fixed time $t\geq 0$. In particular, we see that the under the event $\mathcal{B}^{(k)} \cap \bigcap_{j=1}^{k-1} (\mathcal{C}^{(j)})^{\complement}$, the law of the open ASEP with light particles at time $C_2 N + C_3 k\sqrt{N}$ is, up to the location of the light particles, given by $\mu_N$. Since we condition on the event $\mathcal{B}^{(k)}$ in \eqref{eq:Cvent}, the same arguments as in the proof of Lemma \ref{lem:HittingWindow} and the strong Markov property yield that $C_3$ can be chosen such that for some constants $c_3,C_4>0$ 
\begin{equation}\label{eq:AlmostCvent}
    \mathbf{P}\left(  \lc_r( C_2 N + C_3 (k+1) \sqrt{N} -1 ), \lc_r( C_2 N + C_3 (k+1) \sqrt{N} - 1 )  \leq C_4    \, \big| \,  \mathcal{B}^{(k)} , (\mathcal{C}^{(j)})^{\complement} \, \forall j <k \right) \geq c_3 
\end{equation} for all $k\in \lBr N \rBr$. The claim \eqref{eq:Cvent} follows now from \eqref{eq:AlmostCvent} as any pair of configurations which differ only at locations $\lBr \lceil C_4 \rceil \rBr$ can be coupled in time $1$ to coalesce with positive probability. 
Combining now \eqref{eq:Bvents1}, \eqref{eq:Bvents2}  and \eqref{eq:Cvent}, there exists a constant $C_5>0$ such that
\begin{align*}
   \mathbf{P}\big(  \tau(C_2 N + C_3 C_5 \sqrt{N}\log(N))&=\tau^{\prime}(C_2 N + C_3 C_5 \sqrt{N}\log(N)) \big)  \\
   &\geq  1-   \mathbf{P}\left( (\mathcal{C}^{(k)})^{\complement} \text{ for all } k\leq C_5 \log(N)\right) \\
   & \geq 1 - c_2^{C_5\log(N)} - \mathbf{P}( \mathcal{B}^{(\lceil C_5\log(N) \rceil)})  \\
   &\geq 1 - \frac{2}{N} 
\end{align*}
for all $N$ sufficiently large, allowing us to conclude. 
\end{proof}

We have now all tools to bound the mixing time of the open ASEP with light particles in the fan region of the high density phase. 

\begin{proof}[Proof of Theorem \ref{thm:MainMixingTime}] We will only consider the high density phase as the results for the low density phase follow  by the particle-hole symmetry.
Note that the lower bound on the mixing time is immediate from Lemma~\ref{lem:LightProjection} and Theorem~1.5 in \cite{gantert2023mixing}, which states a lower bound on the mixing times for the standard open ASEP in the high density phase.  For the upper bound on the mixing time, let $(\tau(t))_{ t\geq 0}$ and  $(\tau^{\prime}(t))_{ t\geq 0}$ be two open ASEPs with $r$ many light particles. 
Let $\varepsilon \in (0,1)$ and let $(\eta(t))_{t \geq 0}$ be a standard open ASEP with the same parameters $(q,\alpha,\beta,\gamma,\delta)$. Then using Lemma \ref{lem:LightProjection} and Theorem \ref{thm:MixingTimesOpenASEP},  there exists a constant $C_0>0$  such that for any $\varepsilon>0$ and any pair of initial configurations $\tau(0)$ and $\tau^{\prime}(0)$, under the basic coupling $\mathbf{P}$, 
\begin{equation*}
  \liminf_{N \rightarrow \infty }  \mathbf{P}( \tau_x(C_0 N) = \tau^{\prime}_x(C_0 N)=\eta(C_0 N) \text{ for all } x\in \lBr N \rBr \text{ with }  \tau_x(C_0 N),\tau^{\prime}_x(C_0 N) \in \{0,1\} ) \geq 1 - \frac{\varepsilon}{2}
\end{equation*} 
In particular, up to an error of $\varepsilon/2$, the law of the two open ASEPs $(\tau(t))_{ t\geq 0}$ and  $(\tau^{\prime}(t))_{ t\geq 0}$ agrees at time $C_0 N$ outside of the locations of light particles with the stationary measure $\mu_N$ of a standard open ASEP. By Lemma \ref{lem:Premixing},  there exists  some $C_1>0$ such that for any $\varepsilon>0$
\begin{equation}\label{eq:CouplingResult}
     \liminf_{N \rightarrow \infty }  \mathbf{P}\left(  \tau((C_0+C_1) N) = \tau^{\prime}((C_0+C_1) N)  \right) \geq 1- \varepsilon . 
\end{equation} 
Since $\varepsilon>0$ as well as $\tau(0)$ and $\tau^{\prime}(0)$ were  arbitrary, a standard argument to convert \eqref{eq:CouplingResult} into an upper bound on the mixing times -- see Corollary 2.4 in \cite{gantert2023mixing} --  finishes the proof. 
\end{proof}

\section{Density at the bulk in the coexistence line}\label{sec:proof of coexistence line density}
In this section, we show Proposition \ref{prop:coexistence} regarding the density of the stationary measure of the standard open ASEP on the coexistence line $A=C>1$. 
As a consequence, we prove Theorem \ref{thm:mainCoexistence} on the location of a light particle on the coexistence line. 


We first recall the following result on the ``macroscopic'' density profile of standard open ASEP. 

\begin{theorem}[c.f.\ Theorem 1.6 in \cite{wang2023askey}]\label{thm:HeightFunctionCoexistence} Assume \eqref{eq:conditions qABCD} and let $ABCD\notin\{q^{-l}:l\in\NN_0\}$.
Consider the coexistence line $A=C>1$. Let $\eta^{(N)}\sim\mu_N$. 
Then we have that for all $y \in [0,1]$,
\begin{equation}\label{eq:macroscopic density profile}
\frac{1}{N}\sum_{x=1}^{\lceil yN \rceil}\eta^{(N)}(x)\Longrightarrow \frac{Ay + \min(y,U)(1-A)}{1+A} , 
\end{equation}
where $U$ is a uniform random variable on $[0,1]$, and we use  $\Longrightarrow$ to denote weak convergence.
\end{theorem}
\begin{remark}\label{rmk:shock}
    We mention that \cite[Theorem 1.6]{wang2023askey} shows convergence in finite dimensional distribution, which is slightly stronger than the version above (at a single location $y$). The right hand side of \eqref{eq:macroscopic density profile}, i.e. the macroscopic density profile, has the following physical interpretation:  
   the density is a constant $1/(1+A)$ over $[0,U)$, and a different constant $A/(1+A)$ over $[U,1]$, where $U$ is a random variable uniformly distributed on $[0,1]$, representing the location of the shock. This can be read by re-writing the right hand side of \eqref{eq:macroscopic density profile} as
  \[
  \frac1{1+A}\min(y,U) + \frac A{1+A}(y-\min(y,U)).
  \]  
\end{remark}


Next, we recall the following result on the local convergence of the standard ASEP on the integers.

\begin{theorem}[c.f.\ Theorem 2 in \cite{bahadoran2006convergence}]\label{thm:LocalConvergence}
Let $(\zeta(t))_{t \geq 0}$ denote the ASEP on the integers.  We consider a sequence of initial configurations $\zeta^{(N)}\in \{0,1\}^{\mathbb{Z}}$ for $N \in \mathbb{N}$. Suppose that there exists some $\rho\in (0,1)$ such that as $N\rightarrow\infty$,
\begin{equation}\label{eq:LocalRhoZ}
\frac{1}{N} \sum_{x \in \Z} \psi(x/N) \zeta^{(N)}_x \Longrightarrow \rho\int  \psi(x) \textup{d} x
\end{equation} for all almost everywhere continuous, bounded test functions $\psi \colon \mathbb{R} \rightarrow \mathbb{R}$ of bounded support. Then we have that for every $t_0>0$ and $x_0 \in \mathbb{R}$, 
\begin{equation}
\lim_{N \rightarrow \infty} \mathbb{P}\left( \zeta_{\lceil x_0 N \rceil }(t_0 N) = 1 \, \big| \, \zeta(0)=\zeta^{(N)}\right) = \rho  . 
\end{equation}
\end{theorem}

A similar local convergence result holds for the open ASEP.
\begin{lemma}\label{lem:LocalConvergecenOpenASEP}  Let $(a_N)_{N=1,2,\dots}$ be a sequence of integers such that $a_N\in\llbracket N\rrbracket$ and $a_N/N\rightarrow\theta$ for some $\theta\in(0,1)$. Fix some $\rho\in (0,1)$ and $\varepsilon \in (0,\min(\theta,1-\theta))$. Assume that for each $N\in\N$, there is an initial configuration $\eta^{(N)}$ for the open ASEP on $\lBr N \rBr$.  Assume that as $N\rightarrow\infty$,
\begin{equation}\label{eq:LocalRhoopenASEP}
\frac{1}{N} \sum_{x \in [-\varepsilon N,\varepsilon N]} \psi(x/N) \eta^{(N)}_{a_N + x} \Longrightarrow 2 \varepsilon \rho \int  \psi(x) \textup{d} x
\end{equation} for all almost everywhere continuous, bounded test functions $\psi \colon \mathbb{R} \rightarrow \mathbb{R}$ with $\textup{supp}(\psi) \subseteq [-\varepsilon,\varepsilon]$. Then we have
\begin{equation}\label{eq:ConvergenceopenASEP}
\lim_{N \rightarrow \infty} \mathbb{P}\left( \eta_{a_N}(\varepsilon N/4) = 1 \, \big| \, \eta(0)=\eta^{(N)}\right) = \rho . 
\end{equation}
\end{lemma}
\begin{proof} 
For each $N\in\N$, we first restrict the initial configuration $\eta^{(N)}$ to  $[(\theta-\varepsilon) N,(\theta+\varepsilon) N]$, then extend it to a (random) configuration ${\zeta}^{(N)} \in \{ 0,1 \}^{\Z}$ by setting
\begin{equation}
{\zeta}^{(N)}_x := \begin{cases}
\eta^{(N)}_x & \text{ if }  x \in [(\theta-\varepsilon) N,(\theta+\varepsilon) N]  \\
Y_x   & \text{ otherwise,}
\end{cases}
\end{equation} where $(Y_x)_{x \in \Z}$ is a family of i.i.d.\ Bernoulli-$\rho$-random variables. Using \eqref{eq:LocalRhoopenASEP}, we see that ${\zeta}^{(N)}$ satisfies the assumption \eqref{eq:LocalRhoZ} in Theorem \ref{thm:LocalConvergence}. Thus, in order to conclude \eqref{eq:ConvergenceopenASEP}, it suffices to show that we can couple the ASEP on the integers $(\zeta(t))_{t \geq 0}$ and the open ASEP $(\eta(t))_{t \geq 0}$ such that
\begin{equation}\label{eq:GoodCoupling}
\lim_{N \rightarrow \infty}\mathbf{P}\left( {\eta}_{a_N}(\varepsilon N/4) ={\zeta}_{a_N}(\varepsilon N/4) \, \big| \, \eta(0)=\eta^{(N)} \text{ and } \zeta(0)=\zeta^{(N)}  \right) =1 . 
\end{equation}
Using Lemma \ref{lem:CouplingOpenZ} for $t=\varepsilon N/4$, as well as $a=(\theta-\varepsilon) N$ and $b=(\theta+\varepsilon) N$, we conclude.
\end{proof}

We have now all tools to show Proposition~\ref{prop:coexistence}.

\begin{proof}[Proof of Proposition~\ref{prop:coexistence}]
We first assume $\theta\in(0,1)$. At the end we prove the result for $\theta\in\{0,1\}$.
By Skorokhod's representation theorem, the sequence of random variables $\eta^{(N)}\sim\mu_N$ taking values in $\{0,1\}^N$ can be defined on the same probability space, such that \eqref{eq:macroscopic density profile} is almost sure convergence. 
Let $\varepsilon>0$, we define the events  
\begin{equation}
 A_{\varepsilon} := \Big\{ U \in [0,\theta-\varepsilon ] \Big\} \quad \text{ and } \quad B_{\varepsilon} := \Big\{ U \in [\theta+\varepsilon ,1] \Big\} . 
\end{equation} 
By Theorem \ref{thm:HeightFunctionCoexistence}, on the event $A_{\varepsilon}$ we have
\begin{equation}
 \frac{1}{N} \sum_{x \in [-\varepsilon N/2,\varepsilon N/2] } \psi(x/N) \eta^{(N)}_{a_N + x}  \overset{\mbox{a.s.}}{\longrightarrow}  \frac{\varepsilon A}{1+A} \int \psi(x) \textup{d} x
\end{equation} whereas on the event $B_{\varepsilon}$,  we have
\begin{equation}
 \frac{1}{N} \sum_{x \in [-\varepsilon N/2,\varepsilon N /2] } \psi(x/N) \eta^{(N)}_{a_N + x}  \overset{\mbox{a.s.}}{\longrightarrow} \frac{ \varepsilon }{1+A} \int  \psi(x) \textup{d} x 
\end{equation}  for almost everywhere continuous, bounded test functions $\psi \colon \mathbb{R} \rightarrow \mathbb{R}$ with $\textup{supp}(\psi) \subseteq [-\varepsilon,\varepsilon]$.
This can be seen by approximating $\psi$ both from above and from below by piecewise constant functions.

We then consider the sequence of open ASEPs on $\llbracket N\rrbracket$ starting from initial configurations $\eta^{(N)}$.
From Lemma \ref{lem:LocalConvergecenOpenASEP}, we see that 
\begin{align*}
\lim_{N \rightarrow \infty} \mathbb{P}\left( \eta_{a_N}(\varepsilon N/4)=1 \, \big| \, A_{\varepsilon} , \eta (0) = \eta^{(N)} \right) &= \frac{A}{1+A},   \\ 
\lim_{N \rightarrow \infty} \mathbb{P}\left( \eta_{a_N}(\varepsilon N/4)=1 \, \big| \, B_{\varepsilon}, \eta (0) = \eta^{(N)} \right) &= \frac{1}{1+A}  . 
\end{align*}
In view of $\eta(\varepsilon N/4)\sim \mu_N$ and $\eta (0)= \eta^{(N)} \sim \mu_N$, we have
\[
\lim_{N\rightarrow\infty}\mathbb{P}\left(\eta^{(N)}_{a_N}=1, A_{\varepsilon}\cup B_{\varepsilon}\right)=
\mathbb{P}(A_{\varepsilon})\frac{A}{1+A}+\P(B_{\varepsilon})\frac{1}{1+A}.
\]
Since $\mathbb{P}(A_{\varepsilon})= \theta-\varepsilon$  and $\P(B_{\varepsilon})= 1 - \theta-\varepsilon$, taking $\varepsilon\rightarrow0$ gives the desired result for $\theta \in (0,1)$. 


Finally we show the result for $\theta\in\{0,1\}$. We only consider the case $\theta=0$, as $\theta=1$ can be proved similarly.
Assume that $1\leq a_N\leq N$ and $a_N/N \rightarrow 0$ as $N \rightarrow \infty$. We want to show that
\be\label{eq:ogerv}
\lim_{N\rightarrow\infty}\mu_N(\tau_{a_N}=1)=\frac{1}{1+A}.
\ee

In equation \eqref{eq:simple relation} in Proposition \ref{prop:simple relation}, notice that the left-hand side is strictly positive since the measure $\mt_{N,1}$ is irreducible. Therefore the set
\[
\{\mu_N(\tau_k=1)-\mu_N(\tau_l=1):1\leq k<l\leq N\}
\]
is contained either in $\mathbb{R}_+$ or in $\mathbb{R}_-$. By our result for $\theta\in(0,1)$, it is contained in $\mathbb{R}_-$. 
By Proposition~\ref{prop:boundary density}, on the coexistence line $A=C>1$ we have $\ld=1/(1+A)$. Hence
\be\label{eq:fow}
\liminf_{N\rightarrow\infty}\mu_N(\tau_{a_N}=1)\geq\lim_{N\rightarrow\infty}\mu_N(\tau_{1}=1)=\ld=\frac{1}{1+A}.
\ee
For the reversed inequality, for any $\ep\in(0,1)$ we consider another sequence $(a_N')$ satisfying $a_N\leq a_N'\leq N$ and $a_N'/N\rightarrow\ep$.  By the above observation and the result for $\ep$,
\[
\limsup_{N\rightarrow\infty}\mu_N(\tau_{a_N}=1)\leq\lim_{N\rightarrow\infty}\mu_N(\tau_{a'_N}=1)= (1-\ep)\frac{1}{1+A}+\ep\frac{A}{1+A}.
\]
Taking $\ep\rightarrow0$ we have
\be\label{eq:rev}
\limsup_{N\rightarrow\infty}\mu_N(\tau_{a_N}=1)\leq\frac{1}{1+A}.
\ee
Combining \eqref{eq:fow} and \eqref{eq:rev} we have  \eqref{eq:ogerv}. The proof is concluded.
\end{proof}

We record the following consequence of Proposition \ref{prop:coexistence}, which is of independent interest.

\begin{corollary}
Assume \eqref{eq:conditions qABCD} and $ABCD\notin\{q^{-l}:l\in\NN_0\}$. Consider the coexistence line $A=C>1$. Assume that  $(a_N)_{N=1,2,\dots}$ is a sequence of integers such that 
\begin{equation}\label{eq:LiggettCoex}
    \min(a_N,N-a_N) \rightarrow \infty,
\end{equation} while $a_N/N \rightarrow \theta$ for some $\theta \in [0,1]$. We fix $M\in \N$ and let $\mu_N\llbracket a_N,a_N+M\rrbracket$ denote the measure $\mu_N$ projected onto the interval $\lBr a_N,a_N+M \rBr$.  Then as $N\rightarrow\infty$, we have
    \begin{equation}
    \mu_N\llbracket a_N,a_N+M\rrbracket \Longrightarrow (1-\theta)\textup{Bern}_M\left( \frac{1}{1+A}\right) +\theta \textup{Bern}_M\left(\frac{A}{1+A}\right) , 
    \end{equation} where $\textup{Bern}_M(\rho)$ denotes the Bernoulli-$\rho$-product measure on $\{0,1\}^{M}$ for $\rho \in [0,1]$.  
\end{corollary}
\begin{proof}
    In Theorem 3.29 of \cite{liggett1999stochastic}, Liggett shows that there exists a probability measure $\chi$ on $[0,1]$ such that for all $a_N$ which satisfy \eqref{eq:LiggettCoex}, we get 
     \begin{equation}
       \mu_N\llbracket a_N,a_N+M\rrbracket \Longrightarrow \int \textup{Bern}_M(\rho) \chi( \textup{d} \rho ) ;
    \end{equation}
 see Proposition A.2 in \cite{nestoridi2023approximating} when $q>0$. The result then follows from Proposition \ref{prop:coexistence}.
\end{proof}

We deduce now Theorem \ref{thm:mainCoexistence} from Proposition \ref{prop:coexistence} and Proposition \ref{prop:simple relation}.
\begin{proof}[Proof of Theorem \ref{thm:mainCoexistence}]
On the coexistence line $A=C>1$, assuming $ABCD\notin\{q^{-l}:l\in\NN_0\}$, we want to show that under the measure $\mt_{N,1}$, the distribution of $\lc_1/N$ weakly converges as $N\rightarrow\infty$ to the uniform distribution on $[0,1]$. We only need to show that, for any $0<\theta_1<\theta_2<1$, 
\begin{equation}
    \label{eq:to prove in proof of main theorem coexistence}
\lim_{N\rightarrow\infty}\mt_{N,1}\lb \theta_1\leq\lc_1/N\leq\theta_2\rb=\theta_2-\theta_1.
\end{equation}
By Proposition \ref{prop:simple relation}, taking $k=\lceil\theta_1 N\rceil$ and $l=\lfloor\theta_2 N\rfloor$, we have
\begin{equation} \label{eq:vaerer}
\mt_{N,1}\lb \theta_1\leq\lc_1/N\leq\theta_2\rb=\frac{\m_{N+1}\lb\tau_{\lceil\theta_1 N\rceil}=1\rb-\m_{N+1}\lb\tau_{\lfloor\theta_2 N\rfloor+1}=1\rb}{\m_{N+1}(\tau_1=1)-\m_{N+1}(\tau_{N+1}=1)}.
\end{equation}
We recall Proposition \ref{prop:coexistence}: for any $\theta\in[0,1]$ and any sequence $(a_N)_{N=1,2,\dots}$ satisfying $a_N/N\rightarrow\theta$, 
\[ 
\lim_{N\rightarrow\infty}\mu_N(\tau_{a_N}=1)= (1-\theta)\frac{1}{1+A}+\theta\frac{A}{1+A}.
\]
We take $\theta \in \{ 0,1,\theta_1,\theta_2 \}$, respectively, in the above equation. In view of \eqref{eq:vaerer}, we obtain \eqref{eq:to prove in proof of main theorem coexistence}. This finishes the proof.
\end{proof}

\section*{Declarations}
\subsection*{Funding} Z.Y.~was partially supported by Ivan Corwin’s NSF grant DMS-1811143 as well as the Fernholz Foundation’s `Summer Minerva Fellows' program.
\subsection*{Conflicts of interests/Competing interests} The authors have no competing interests to declare that
are relevant to the content of this article. 
\subsection*{Data availability statements} Data sharing not applicable to this article as no datasets were generated
or analysed during the current study.

\bibliographystyle{goodbibtexstyle}
\bibliography{twospecies}

\end{document}